%% file: pr_ncvx.tex
\newif\ifMS
\MSfalse

\ifMS
\documentclass[twocolumn]{IEEEtran}
\else
\documentclass[a4paper]{article}
\usepackage[margin=1in]{geometry}
\fi

\usepackage{mathtools}
\usepackage{amsmath,amssymb,amsfonts}%
\usepackage{amsthm}%
\usepackage{mathrsfs}%
\usepackage{textcomp}%
\usepackage{manyfoot}%
\usepackage{xparse}
\usepackage[normalem]{ulem}
\usepackage{caption}
\usepackage{subcaption}
\usepackage{hyperref}
\hypersetup{
	colorlinks,
	linkcolor=blue,
	citecolor=blue,
	urlcolor=blue,
}
\usepackage[capitalize]{cleveref}

\DeclarePairedDelimiterXPP{\opnorm}[1]{}{\lVert}{\rVert}{_{\mathrm{op}}}{#1}
\DeclarePairedDelimiterXPP{\nucnorm}[1]{}{\lVert}{\rVert}{_{*}}{#1}

\newcommand{\st}{\text{ s.t.\ }}
\newcommand{\PT}{\scrP_{\scrT}}
\newcommand{\PTp}{\scrP_{\scrT^\perp}}

\newcommand{\PTpmat}{P_{\scrT^\perp}}

\usepackage{import}
\import{./}{definitions}

\theoremstyle{plain}%
\newtheorem{theorem}{Theorem}%  meant for continuous numbers
\newtheorem{lemma}{Lemma}% 

\theoremstyle{definition}%
\newtheorem{assumption}{Assumption}%
\Crefname{assumption}{Assumption}{Assumptions}

\newcommand{\herms}{\mathbf{H}}
\newcommand{\symms}{\mathbf{S}}

\newcommand{\real}{\operatorname{Re}}

\newcommand{\moment}{m}
\newcommand{\fourthmoment}{\moment_4}

\newcommand{\Pmat}{P_{U}}
\newcommand{\Ppmat}{P_U^\perp}
\newcommand{\PXp}{P_{X}^\perp}
\newcommand{\T}{\scrT}
\newcommand{\Tp}{\scrT^\perp}
\newcommand{\HT}{H_{\scrT}}
\newcommand{\HTp}{H_{\scrT^\perp}}

\newcommand{\fundingack}{This work was supported by the Swiss State Secretariat for Education, Research, and Innovation (SERI) under contract MB22.00027 during the author's time at EPFL.}
\begin{document}

\title{Phase retrieval and matrix sensing via benign and overparametrized nonconvex optimization}
\author{Andrew D.\ McRae%
	\ifMS \thanks{Manuscript received May 19, 2025; revised November 12, 2025. \fundingack{}}\fi %
	\thanks{The author is with CERMICS, CNRS, ENPC, Institut Polytechnique de Paris, Marne-la-Vallée, France. E-mail: \href{mailto:andrew.mcrae@enpc.fr}{andrew.mcrae@enpc.fr}.
		\ifMS \else This document is the accepted manuscript version of \url{https://doi.org/10.1109/TIT.2026.3684748}. © 2026 IEEE.  Personal use of this material is permitted.  Permission from IEEE must be obtained for all other uses, in any current or future media, including reprinting/republishing this material for advertising or promotional purposes, creating new collective works, for resale or redistribution to servers or lists, or reuse of any copyrighted component of this work in other works.
		\fi
	}%
}

\maketitle

\ifMS
\markboth{Journal of \LaTeX\ Class Files,~Vol.~14, No.~8, August~2015}%
{Shell \MakeLowercase{\textit{et al.}}: Bare Demo of IEEEtran.cls for IEEE Journals}
\fi

\begin{abstract}
We study a nonconvex optimization algorithmic approach to phase retrieval and the more general problem of semidefinite low-rank matrix sensing.
Specifically, we analyze the nonconvex landscape of a quartic Burer-Monteiro factored least-squares optimization problem.
We develop a new analysis framework, taking advantage of the semidefinite problem structure, to understand the properties of second-order critical points---specifically, whether they (approximately) recover the ground truth matrix.
We show that it can be helpful to (mildly) overparametrize the problem, that is, to optimize over matrices of higher rank than the ground truth.
We then apply this framework to several well-studied problem instances:
in addition to recovering existing state-of-the-art phase retrieval landscape guarantees (without overparametrization),
we show that overparametrizing by a factor at most logarithmic in the dimension allows recovery with optimal statistical sample complexity and error for the problems of (1) phase retrieval with sub-Gaussian measurements
and (2) more general semidefinite matrix sensing with rank-1 Gaussian measurements.
Previously, such statistical results had been shown only for estimators based on semidefinite programming.
More generally, our analysis is partially based on the powerful method of convex dual certificates, suggesting that it could be applied to a much wider class of problems.
\end{abstract}

\ifMS
\fi

\section{Introduction and result highlights}
\label{sec:intro}
\ifMS\IEEEPARstart{T}{his}
\else This
\fi
paper considers the problem of estimating of positive semidefinite (real or complex) $d \times d$ matrix $Z_*$ from (real) measurements of the form
\[
	y_i \approx \ip{A_i}{Z_*}, \quad i = 1, \dots, n,
\]
where $A_1, \dots, A_n$ are known positive semidefinite matrices,
and $\ip{\cdot}{\cdot}$ denotes the elementwise Euclidean (Frobenius) matrix inner product.
We will denote $r = \rank(Z_*)$, which we typically assume to be much smaller than the dimension $d$.
This is an instance of the well-studied problem of \emph{low-rank matrix sensing}.
However, our requirement that the measurement matrices $\{A_i\}$ are positive semidefinite is quite particular and, as we will see, adds significant structure to the problem.
Hence we refer to our problem as \emph{semidefinite} low-rank matrix sensing.

A key instance of this problem is \emph{phase retrieval},
where we want to recover a vector $x_*$ from (approximate) magnitude measurements of the form $\abs{\ip{a_i}{x_*}}$, where $a_1, \dots, a_n$ are known measurement vectors (on vectors, $\ip{\cdot}{\cdot}$ is the usual real or complex Euclidean inner product).
Phase retrieval arises in many applications, particularly those involving estimation or image reconstruction from optical measurements (where we may observe the intensity but not the phase of an electromagnetic wave). See \Cref{sec:relwork} for further reading.
Phase retrieval can be cast as a semidefinite rank-one matrix sensing problem by noting that $\abs{\ip{a_i}{x_*}}^2 = \ip{a_i a_i^*}{x_* x_*^*}$.

To be more precise, let $\F$ be the set of real or complex numbers (i.e., $\F = \R$ or $\F = \C$).
We denote by $\herms_d$ the set of Hermitian matrices in $\F^{d \times d}$.
We want to recover a rank-$r$ positive semidefinite (PSD) matrix $Z_* \in \herms_d$
from measurements of the form
\begin{equation}
	\label{eq:model_gen}
	\begin{gathered}
	y_i = \ip{A_i}{Z_*} + \xi_i \in \R,\quad i = 1,\dots, n, \quad \text{or} \\
	y = \scrA(Z_*) + \xi,
	\end{gathered}
\end{equation}
where $A_1, \dots, A_n \in \herms_d$ are known PSD matrices,
we denote $y = (y_1, \dots, y_n) \in \R^n$, $\xi = (\xi_1, \dots, \xi_n) \in \R^n$,
and the linear operator $\scrA \colon \herms_d \to \R^n$ is defined by
\[
	\scrA(S) = \begin{bmatrix*}\ip{A_1}{S} \\ \vdots \\ \ip{A_n}{S} \end{bmatrix*}.
\]
(note that the inner product of Hermitian matrices is always real).

In the phase retrieval problem, for unknown $x_* \in \F^d$ and known measurement vectors $a_1, \dots, a_n \in \F^d$, we use the model \eqref{eq:model_gen} with
\begin{equation}
	\label{eq:pr_model}
	Z_* = x_* x_*^*, \quad \text{and} \quad A_i = a_i a_i^*, \quad i = 1,\dots, n.
\end{equation}
If, as is often the case in practice, $x_*$ is real ($\F = \R$) but the measurements are complex,
we can make everything real by taking $A_i = \real(a_i a_i^*)$.

Given measurements of the form \eqref{eq:model_gen},
a natural question is how to estimate $Z_*$ from the data $\{(A_i, y_i)\}_{i=1}^n$.
Many algorithms exist, particularly for the phase retrieval model \eqref{eq:pr_model} (see the surveys and other selected references in \Cref{sec:relwork}).
However, the vast majority of theoretical guarantees are for the simplest and idealized case of phase retrieval with Gaussian measurements (i.e., taking $a_1, \dots, a_n$ to be standard Gaussian random vectors).
Outside this one well-studied case, the best theoretical guarantees (in terms of statistical performance) are generally for estimators based on some variant of the following convex semidefinite program (SDP):
\begin{equation*}
	\tag{PhaseLift}
	\label{eq:phaselift}
	\min_{Z \succeq 0}~\norm{y - \scrA(Z)}^2.
\end{equation*}
This is one version %\footnote{The original formulation includes the trace of $Z$ to promote low rank, but this was shown to be unnecessary by \cite{Candes2013,Demanet2014}, whose work we will later build on.}%
of the PhaseLift program introduced by \cite{Chai2011,Candes2013b} as a convex relaxation for phase retrieval.
Although \eqref{eq:phaselift} is convex and thus amenable to theoretical analysis, the feasible set has order $d^2$ degrees of freedom.
The ground truth $Z_*$ has order $rd$ degrees of freedom,
which is far lower if the rank $r$ is small (e.g., for phase retrieval, $r = 1$).
Thus, for computational and storage purposes, directly solving \eqref{eq:phaselift} with standard solvers is not very practical when $d$ is large.

Apart from the SDP approach of \eqref{eq:phaselift},
most algorithms and/or their theoretical guarantees (when those exist) are quite complicated,
with elements such as special cost functions, careful initialization and/or truncation schemes, or intricate analyses of iterative algorithms;
most of these theoretical analyses have suboptimal statistical properties and/or are forced to assume Gaussian measurements for simplicity.
Again, see \Cref{sec:relwork} for further reading.

We want an algorithmic approach that has the same conceptual simplicity and nice theoretical properties as \eqref{eq:phaselift} but scales better computationally as the dimension $d$ increases.
We therefore instead focus on a smooth low-rank Burer-Monteiro factored version of \eqref{eq:phaselift}.
Given a search rank parameter $p$,
the resulting problem is
\begin{equation*}
	\tag{BM-LS}
	\label{eq:opt_gen}
	\begin{aligned}
		&\min_{X \in \F^{d \times p}}~f_p(X), \quad \text{where} \\
		&\qquad f_p(X) = \norm{y- \scrA(X X^*)}^2.
	\end{aligned}
\end{equation*}
%Although, from now on, we do not consider the nonsmooth low-rank problem \eqref{eq:opt_lorank},
%we can still draw conclusions about it via this parametrization.
%In particular, via the work of \cite{Ha2020}, all guarantees that we prove in this paper for second-order critical points of \eqref{eq:opt_gen} will also hold for all local minima of \eqref{eq:opt_lorank}.

Two natural questions arise, answering which will be the main focus of this paper:
\begin{itemize}
	\item Although the problem \eqref{eq:opt_gen} is smooth ($f_p$ is a quartic polynomial in the elements\footnote{The real and imaginary parts, in the complex case.} of the variable $X$), it is nonconvex,
	and thus, potentially, local algorithms could get stuck in spurious local optima.
	Is this a problem?
%	Some existing works overcome this in some cases with careful initialization schemes and/or algorithmic analyses, but we would like something simpler and more general.
	\item How do we choose the estimation rank $p$? The obvious choice is $p = r = \rank(Z_*)$ if it is known, but is this the best choice?
	For practical reasons we want $p$ to be small (e.g., constant or at least $\ll d$).
\end{itemize}

Nonconvex problems of the form \eqref{eq:opt_gen} have indeed been well studied in the low-rank matrix sensing and optimization literature,
and there are many positive results showing that such problems can have a \emph{benign landscape}:
that every local minimum (or even second-order critical point\footnote{A point where the gradient is zero and the Hessian is positive semidefinite; we unpack this condition in more detail in \Cref{sec:basic_calcs}.})
is ``good'' in some sense (e.g., it is a global optimum or at least is close to the ground truth).
This fits our purposes well, because it is well-known (and even rigorously proved---see, e.g., \cite{Curtis2018,Lee2019b} for merely two of many such results) that local search methods such as gradient descent or trust-region algorithms will find second-order critical points of problems like \eqref{eq:opt_gen}.

However, the vast majority of existing results (see \Cref{sec:relwork}) of this character make strong assumptions about the measurement operator $\scrA$:
specifically, that it has a \emph{restricted isometry property} (RIP) in the sense that (up to rescaling) $\frac{1}{n} \norm{\scrA(S)}^2 \approx \normF{S}^2$ for all $S \in \herms_d$ with low rank ($\normF{\cdot}$ denotes the matrix Frobenius norm, i.e., the elementwise Euclidean norm).
However, RIP is often an unreasonably strong condition.
For example, we will see soon that it does not hold for phase retrieval without an unreasonably large number of measurements.
Therefore, other approaches are needed.

In this paper, we develop a novel analysis framework of the nonconvex landscape of \eqref{eq:opt_gen}.
This framework does not require RIP and exploits the semidefinite problem structure.
We then use this framework to show that certain popular instances of \eqref{eq:opt_gen} do indeed have a benign landscape in that every second-order critical point either recovers (in the absence of noise) the ground truth and is globally optimal
or at least (with noise) gives a statistically accurate estimator of the ground truth.
In particular, our results reveal the benefits of mild \emph{rank overparametrization}, that is, setting the search parameter $p$ to be strictly larger than the ground truth rank $r$; we obtain a benign landscape with statistically optimal sample complexity with $p$ chosen to be at most of order $r \log d$.
Previous state-of-the-art statistical results for these problems directly studied some version of the SDP \eqref{eq:phaselift}.

In the rest of this introduction, we give a tour of some of the challenges faced, some of the concrete implications of our analysis, and some future research directions for which we believe our framework will be helpful.

\subsection{Phase retrieval with rank-one optimization}
\label{sec:intro_pr_r1}
To see better the challenges we face in trying to understand the landscape of \eqref{eq:opt_gen},
we begin with the simple and well-studied phase retrieval model \eqref{eq:pr_model} with Gaussian measurements.
As the target matrix has rank $1$, it is natural to consider \eqref{eq:opt_gen} with rank parameter $p = 1$.
We then obtain the problem
\begin{equation*}
	\tag{PR-LS}
	\label{eq:opt_pr_r1}
	\begin{gathered}
		\min_{x \in \F^d}~f(x), \qquad \text{where} \\
		f(x) = \norm{ y - \scrA(x x^*) }^2 = \sum_{i=1}^n (y_i - \abs{\ip{a_i}{x}}^2)^2.
	\end{gathered}
\end{equation*}
This objective function was proposed for phase retrieval by \cite{GuizarSicairos2008} (as a special case of a larger family of loss functions) before being studied in more detail by \cite{Candes2015}.

One might hope to use something like the restricted isometry property (RIP) mentioned before to analyze \eqref{eq:opt_pr_r1}.
For example, if the measurement vectors $a_1, \dots, a_n$ are chosen as independent and independently distributed (i.i.d.) real or complex standard Gaussian random vectors,
one can easily calculate that, for $S \in \herms_d$, $\scrA$ in expectation satisfies
\begin{equation}
	\label{eq:expectation_gauss}
	\begin{gathered}
		\E \frac{1}{n} \scrA^* \scrA(S) = \fourthmoment S + (\tr S) I_d \quad \\
		\Big \Downarrow \\
		\E \frac{1}{n} \norm{\scrA(S)}^2 = \fourthmoment \normF{S}^2 + \tr^2(S),
	\end{gathered}
\end{equation}
where $\scrA^* \colon \R^n \to \herms_d$ is the adjoint of $\scrA$ given by $\scrA^*(z) = \sum_{i=1}^n z_i A_i$,
and we define $\fourthmoment \coloneqq \E \abs{z}^4 - 1$,
where $z$ is a real or complex standard normal random variable;
with real Gaussian measurements, $\fourthmoment = 2$, and with complex Gaussian measurements, $\fourthmoment = 1$.
Thus, in expectation, up to the trace term and scaling, $\scrA$ is an isometry.

Unfortunately, even with Gaussian measurements, there is no hope of $\scrA$ having RIP without an unreasonably large (at least order $d^2$) number of measurements:
if we take $S = \frac{1}{\norm{a_1}^2} a_1 a_1^*$, which is rank-one and has unit Frobenius norm,
it is clear that\footnote{We write $a \lesssim b$ (equivalently, $b \gtrsim a$) to mean $a \leq c b$ for some unspecified but universal constant $c > 0$. We will similarly, on occasion, write $A \precsim B$ or $B \succsim A$ to denote the semidefinite ordering $c B - A \succeq 0$ for some $c > 0$. We write $a \approx b$ to mean $a \lesssim b$ and $a \gtrsim b$ simultaneously.}
\[
	\frac{1}{n} \norm{\scrA(S)}^2 \geq \frac{\norm{a_1}^4}{n} \gtrsim \frac{d^2}{n}
\]
with high probability. This phenomenon was noted for more general matrix sensing with rank-one measurement matrices in \cite{Cai2015}.

Nevertheless, with more specialized analysis,
we can still say something about the landscape of \eqref{eq:opt_pr_r1}.
Sun et al.\ \cite{Sun2018}, followed by Cai et al.\ \cite{Cai2023}, give positive results when the measurements are Gaussian.
The following theorem, proved in \Cref{sec:proof_gausslike}, is a generalization of their results (here, $\opnorm{\cdot}$ denotes the matrix $\ell_2$ operator norm). We ignore noise for the sake of simplicity and clarity (later results will account for it).
\begin{theorem}
	\label{thm:pr_benign_r1}
	Consider the model \eqref{eq:model_gen} with rank-one $Z_* = x_* x_*^*$, and assume exact measurements, that is, $\xi = 0$.
	\begin{enumerate}
		\item \label{lb:benign_gauss_p1} Suppose $\scrA$ satisfies, for some $\moment, \delta_L, \delta_U > 0$,
		\begin{align*}
			\frac{1}{n} \opnorm{\scrA^* \scrA(x_* x_*^*)} &\leq (1 + \moment + \delta_U) \norm{x_*}^2, \quad \text{and} \\
			\frac{1}{n} \norm{\scrA(x x^* - x_* x_*^*)}^2 &\geq (1 - \delta_L) [\moment \normF{x x^* - x_* x_*^*}^2 \\
			&\qquad + (\norm{x}^2 - \norm{x_*}^2)^2 ]
		\end{align*}
		for all $x \in \F^d$,
		and suppose
		\begin{equation}
			\label{eq:benign_gauss_cond}
			\moment^2 + 2\moment - 2 > 3 (\moment^2 + 2\moment) \delta_L + 2 (\moment + 1) \delta_U.
		\end{equation}
		Then every second-order critical point $x$ of \eqref{eq:opt_pr_r1} satisfies $x x^* = x_* x_*^*$, that is, $x = x_* s$ for some unit-modulus $s \in \F$.
		
		\item \label{lb:benign_gauss_p2} For fixed $x_*$, if $A_i = a_i a_i^*$ (or $A_i = \real(a_i a_i^*)$ if $\F = \R$) for i.i.d.\ real or complex standard Gaussian vectors $a_1, \dots, a_n$,
		then, for universal constants $c_1, c_2 > 0$,
		if $n \geq c_1 d \log d$,
		the conditions of part~\ref{lb:benign_gauss_p1} (with $\moment = \fourthmoment$) are satisfied
		with probability\footnote{Throughout this paper, we state probability bounds in the form $1 - c n^{-2}$ for some $c > 0$, but inspection of the proofs reveals that we can replace $n^{-2}$ with $n^{-\gamma}$ for any constant $\gamma > 0$ with only a change in the other (unspecified) constants depending on $\gamma$.} at least $1 - c_2 n^{-2}$.
	\end{enumerate}
\end{theorem}
Part~\ref{lb:benign_gauss_p2} recovers the result of \cite{Cai2023}.
This was an improvement of the result of \cite{Sun2018}, which required $n \gtrsim d \log^3 d$
(see \Cref{sec:relwork_landscapes} for further discussion and related work).
The arguments in those papers rely on fact that the measurements are Gaussian.
The deterministic condition of part~\ref{lb:benign_gauss_p1} is novel and has the benefit of applying to more general measurement ensembles.
%Part~\ref{lb:benign_gauss_p2} uses concentration inequalities for Gaussian measurements that were already proved and used in \cite{Sun2018}.
%Again, see \Cref{sec:proof_gausslike} for more details.

Note that in the real Gaussian measurement case, where $\fourthmoment = 2$,
condition \eqref{eq:benign_gauss_cond} can be simplified to
\[
	2 \delta_L + \delta_U < 1.
\]
In the complex Gaussian measurement case, where $\fourthmoment = 1$, condition \eqref{eq:benign_gauss_cond} becomes
\[
	9 \delta_L + 4 \delta_U < 1.
\]

Comparing the conditions on $\scrA$ in \Cref{thm:pr_benign_r1} with RIP and the expectation calculation \eqref{eq:expectation_gauss},
note that we have in part maintained the requirement for restricted \emph{lower} isometry ($\frac{1}{n} \norm{\scrA(S)}^2 \gtrsim \normF{S}^2$),
but the restricted \emph{upper} isometry condition ($\frac{1}{n} \norm{\scrA(S)}^2 \lesssim \normF{S}^2$), which is what fails in the given counterexample, has been relaxed.
Instead, we only require (upper) concentration of $\scrA^* \scrA$ on the specific input $x_* x_*^*$,
which is easier to obtain when $x_*$ is fixed independently of the random measurement vectors $a_1, \dots, a_n$.
%This can also be interpreted as approximate (upper) isometry of $\scrA$ on the tangent space $\T \subseteq \herms_d$ of the manifold of rank-one matrices at $Z_* = x_* x_*^*$;
%see \Cref{sec:dualcerts} for a precise definition.

Unfortunately, the conditions of \Cref{thm:pr_benign_r1} are still unreasonably strong in many cases.
The deterministic part~\ref{lb:benign_gauss_p1} is quite sensitive to the (almost-Gaussian) structure of $\scrA$,
and, even in the (often idealized) Gaussian case,
the sample requirement $n \gtrsim d \log d$ of part~\ref{lb:benign_gauss_p2}
(which cannot be improved if we want to satisfy the conditions of part~\ref{lb:benign_gauss_p1})
is suboptimal by a logarithmic factor compared to the best-possible sample complexity $n \approx d$ that is achieved by other methods such as PhaseLift \cite{Candes2013} or various other, more elaborate schemes (see \cite{Dong2023,Fannjiang2020} for an overview).
There is some evidence that the landscape of \eqref{eq:opt_gen} can be benign even with $n \approx d$ Gaussian measurements \cite{SaraoMannelli2020}, but this has not been proved rigorously.

Thus more work is needed to show that ordinary least-squares optimization (without, e.g., worrying about initialization) can solve phase retrieval with flexible measurement assumptions and optimal sample complexity.
%The question thus remains: can we provably solve phase retrieval with optimal sample complexity, using only ordinary local optimization of the smooth least-squares problem (without worrying about sufficiently accurate initialization).

\subsection{The benefits and pitfalls of overparametrization}
\label{sec:intro_overp}
Qualitatively, the condition on $\scrA$ in \Cref{thm:pr_benign_r1}, part~\ref{lb:benign_gauss_p1} can be interpreted as a \emph{condition number} requirement.
Similar requirements appear elsewhere in nonconvex matrix optimization;
in particular, benign landscape results assuming the restricted isometry property (RIP) require the upper and lower isometry constants to be not too different.
See, for example, \cite{Bi2022} and the further references therein.
However, a poor condition number can be mitigated in part by \emph{overparametrizing} the optimization problem, that is, setting the optimization rank $p$ to be strictly larger than the rank $r$ of the ground truth or global optimum.
In matrix sensing under RIP assumptions, this idea appears in the recent works \cite{Ma2023b,Zhang2025c,Zhang2025a}.
The idea also appears in other, structurally quite different examples of low-rank matrix optimization: see \Cref{sec:relwork_overp}.

\begin{figure*}
	\centering
	\begin{subfigure}[t]{0.3\linewidth}
		\centering
		\includegraphics[width=0.9\linewidth]{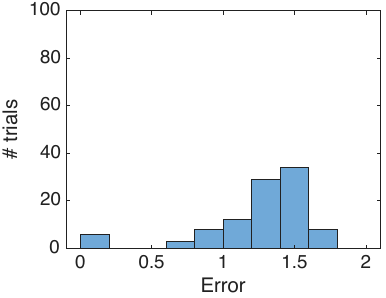}
		\caption{$\F = \R$, $n = 3d$, $p = 1$}
	\end{subfigure}
	\begin{subfigure}[t]{0.3\linewidth}
		\centering
		\includegraphics[width=0.9\linewidth]{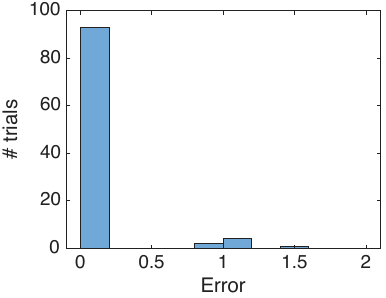}
		\caption{$\F = \R$, $n = 3d$, $p = 2$}
	\end{subfigure}
	\begin{subfigure}[t]{0.3\linewidth}
		\centering
		\includegraphics[width=0.9\linewidth]{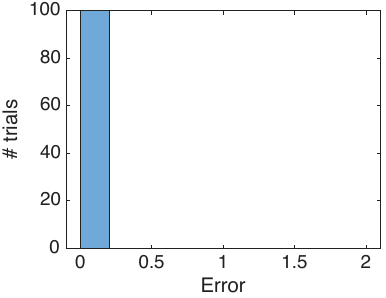}
		\caption{$\F = \R$, $n = 3d$, $p = 3$}
	\end{subfigure} \\[3ex]
%	\caption{100 random trials with $\F = \R$, $d = 50$, and $n = 3d$ Gaussian measurements. For each independent trial, the data are generated randomly, and we attempt to solve \eqref{eq:opt_gen} for each value of $r$ (with independent random initializations).}
%\end{figure*}
%\begin{figure}
%	\centering
	\begin{subfigure}[t]{0.3\linewidth}
		\centering
		\includegraphics[width=0.9\linewidth]{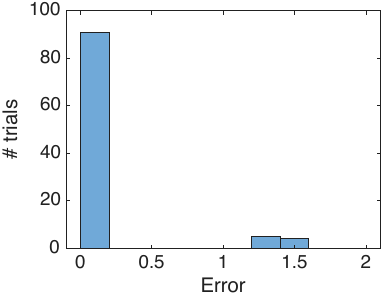}
		\caption{$\F = \C$, $n = 5d$, $p = 1$}
	\end{subfigure}
	\begin{subfigure}[t]{0.3\linewidth}
		\centering
		\includegraphics[width=0.9\linewidth]{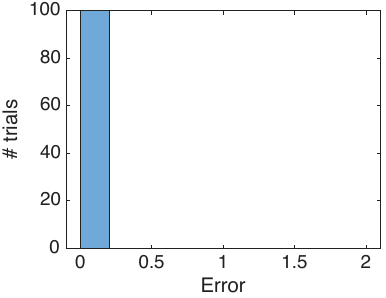}
		\caption{$\F = \C$, $n = 5d$, $p = 2$}
	\end{subfigure}
	\begin{subfigure}[t]{0.3\linewidth}
		\centering
		\includegraphics[width=0.9\linewidth]{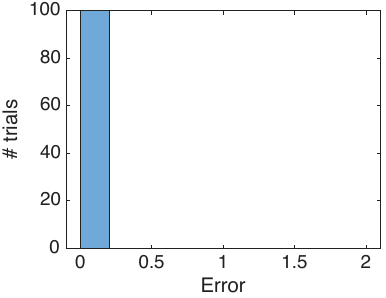}
		\caption{$\F = \C$, $n = 5d$, $p = 3$}
	\end{subfigure}
	\caption{Error histograms of 100 random trials with $d = 50$, $\norm{x_*} = 1$, and standard Gaussian (in $\F$) measurements. For each row of the figure, we generate 100 independent problems, and we then attempt to solve \eqref{eq:opt_gen} for each $p \in \{1, 2, 3\}$ (with independent random initializations). We use a trust-region solver from the Manopt library \cite{Boumal2014a}. For the computed $X \in \F^{d \times p}$, the calculated error is $\min_{\abs{s} = 1} \norm{\xhat - s x_*}$, where $\xhat \xhat^*$ is the closest rank-1 approximation to $X X^*$.}
	\label{fig:exps}
\end{figure*}

The practical benefits of overparametrization are immediately clear from numerical experiments. \Cref{fig:exps} shows error histograms for direct local optimization of \eqref{eq:opt_gen} (with random initialization) for phase retrieval problems with standard Gaussian measurements and small oversampling ratios $n/d$.
The benefits of setting $p > 1$ are clear in both the real and complex cases and are particularly dramatic in the real case.
In both cases, we have a 100\% exact recovery rate for $p = 3$, while smaller values of $p$ fail on some of the same problem instances.

With the analysis framework in this paper,
we can indeed show that overparametrization brings benefits for our problem,
allowing us, in part, to relax the conditions of \Cref{thm:pr_benign_r1}.
The following result is a direct generalization of \Cref{thm:pr_benign_r1} to the overparametrized case.
It has several shortcomings that we will mention and that will be resolved in our later results,
but we include it because it illustrates well the effect of overparametrization and because it applies to small values of $p$ (even $p = 1$).
\newcommand{\lowergaussTR}{\beta}
\newcommand{\lowergaussF}{\alpha}
\begin{theorem}
	\label{thm:pr_benign_overp}
	Consider the model \eqref{eq:model_gen} with rank-one $Z_* = x_* x_*^*$ and $\xi = 0$.
	For rank parameter $p \geq 1$, consider the nonconvex least-squares problem \eqref{eq:opt_gen}.
	Suppose, for some constants $\lowergaussF, \lowergaussTR, L \geq 0$, we have,
	for all $X \in \F^{d \times p}$,
	\begin{equation}
		\label{eq:pr_overp_lb}
		\begin{aligned}
			\frac{1}{n} \norm{\scrA(X X^* - x_* x_*^*)}^2
			&\geq \lowergaussF \normF{X X^* - x_* x_*^*}^2 \\
			&\qquad + \lowergaussTR(\normF{X}^2 - \norm{x_*}^2)^2,
		\end{aligned}
	\end{equation}
	and
	\begin{equation}
		\label{eq:pr_overp_ub}
		\frac{1}{n} \scrA^* \scrA(x_* x_*^*) \preceq L \norm{x_*}^2 I_d.
	\end{equation}
	Then, if
	\begin{equation}
		\label{eq:pr_overp_cond}
		(p+2)\parens*{1 + \frac{\lowergaussTR}{p \lowergaussTR + \lowergaussF}}\lowergaussF > 2 L,
	\end{equation}
	every second-order critical point $X$ of \eqref{eq:opt_gen} satisfies $X X^* = x_* x_*^*$.
\end{theorem}
We prove this in \Cref{sec:proof_gausslike}.
\Cref{thm:pr_benign_r1}, part~\ref{lb:benign_gauss_p1} is indeed a special case of this.
We have again ignored noise for simplicity; in addition, the assumptions as stated are not well-suited for proving optimal error rates with respect to noise. Our other results will handle noise.

It is tempting to say that, as long as $\lowergaussF > 0$, we can always make $p$ large enough so that condition \eqref{eq:pr_overp_cond} is satisfied.
In some cases, this may be true (see, e.g., \cite{Criscitiello2025preprint}).
However, the lower isometry condition \eqref{eq:pr_overp_lb} depends on $p$
in that it must hold for all $X \in \F^{d \times p}$.
For Gaussian measurements,
attempting to prove an inequality like \eqref{eq:pr_overp_lb} with the same methods used by \cite{Sun2018} to prove a similar result for $p = 1$ (see \Cref{lem:gauss_lower_direct} in \Cref{sec:proof_gausslike}) would require $n \gtrsim p d$,
which would defeat any sample-complexity benefit of \Cref{thm:pr_benign_overp} over \Cref{thm:pr_benign_r1}.

For general low-rank matrix sensing problems, the implicit dependence of the lower isometry constants on the optimization rank $p$ is a fundamental limitation of the nonconvex optimization approach.
For example, even with the stronger assumption of RIP, Richard Zhang has given a counterexample (private correspondence; a version appears in \cite{Zhang2025c}) showing that, with excessive overparametrization,
the problem \eqref{eq:opt_gen} may have spurious local minima very far from the ground truth
even when, for $p$ closer to $r$, RIP holds and the landscape is indeed benign.

However, in our case, the fact that both the ground truth $Z_*$ and our measurement matrices $\{A_i\}_i$ are PSD
allows us to overcome this limitation.
%\footnote{This is not the only place we need this; the deterministic landscape analysis underlying all our results (see \Cref{sec:basic_calcs}) also depends on the PSD structure, and it is not clear how to generalize it in a way that retains the benefits of overparametrization.}

\subsection{PSD measurements and universal lower isometry}
\newcommand{\crdvar}{w}
\newcommand{\incohp}{\mu}
\label{sec:intro_dualcerts}
To reap the full benefits of a result like \Cref{thm:pr_benign_overp},
we need the measurement operator $\scrA$ to satisfy lower isometry in a relatively \emph{un}restricted sense:
we want, for any $Z \succeq 0$ (even of high rank),
\[
	\frac{1}{n}\norm{\scrA(Z - Z_*)}^2 \geq \lowergaussF \normF{Z - Z_*}^2
\]
for some $\lowergaussF > 0$.
As discussed above, this is, in general, too much to ask even if $\scrA$ has RIP for suitably small ranks.
However, the PSD structure of $Z$, $Z_*$, and the measurement matrices $\{A_i\}_i$ allows us to do more.
Indeed, this has already been observed and studied (see below) for certain variants of the convex relaxation \eqref{eq:phaselift}
for which there is no rank restriction or penalization.

At a high level, the argument goes as follows.
Given $Z \succeq 0$, we can decompose the error $H = Z - Z_*$ into two components.
We write $H = H_1 + H_2$,
where $H_1$ indeed has low rank (of order $r = \rank(Z_*)$),
and $H_2$ may have large rank but is PSD, that is, $H_2 \succeq 0$ (see, e.g. \Cref{sec:dualcerts} for a principled way to do this).

As $H_1$ has low rank,
we can reasonably hope to show that $\frac{1}{n}\norm{\scrA(H_1)}^2 \gtrsim \normF{H_1}^2$.
Furthermore, as $H_2$ and the matrices $A_i$ are PSD, we have
\begin{align*}
	\frac{1}{\sqrt{n}} \norm{\scrA(H_2)}
	&\geq \frac{1}{n} \sum_{i=1}^n \abs{\ip{A_i}{H_2}} \\
	&= \ip*{ \frac{1}{n} \sum_{i=1}^n A_i }{H_2}.
\end{align*}
The equality holds because $\ip{A_i}{H_2} \geq 0$.
If, for example, $A_i = a_i a_i^*$ for i.i.d.\ standard Gaussian $a_i$,
standard concentration inequalities imply that, when $n \gtrsim d$, $\frac{1}{n} \sum_{i=1}^n A_i \succsim I_d$, in which case we obtain
\[
	\frac{1}{n} \norm{\scrA(H_2)}^2 \gtrsim \tr^2(H_2) = \nucnorm{H_2}^2 \geq \normF{H_2}^2,
\]
where $\tr(\cdot)$ and $\nucnorm{\cdot}$ respectively denote the matrix trace and nuclear norm.%
\footnote{The trace/nuclear norm term appearing here can be interpreted as ``implicit regularization'' arising from the semidefinite problem structure; see the discussion after \Cref{thm:gauss_final_nl}.}

It is thus relatively straightforward to show that, with $H = H_1 + H_2$, $\frac{1}{n}\norm{\scrA(H_1)}^2 \gtrsim \normF{H_1}^2$ and $\frac{1}{n}\norm{\scrA(H_2)}^2 \gtrsim \normF{H_2}^2$.
However, showing that we can ``combine'' these to obtain $\frac{1}{n}\norm{\scrA(H)}^2 \gtrsim \normF{H}^2$
is quite technical and requires additional tools.

\subsubsection{Phase retrieval with general sub-Gaussian measurements}
\label{sec:intro_subG}
We consider, in this work, two separate approaches to proving such lower isometry.
One is based on the work of Krahmer and Stöger \cite{Krahmer2020},
who consider the case of ordinary phase retrieval ($r = 1$).
This analysis framework allows for more general \emph{sub}-Gaussian measurements.
Specifically, we assume the entries of the measurement vectors $a_i$ are i.i.d.\ copies of a random variable $\crdvar$ which we assume to be zero-mean and to satisfy (without loss of generality) $\E \abs{\crdvar}^2 = 1$.
We also assume that $\crdvar$ is sub-Gaussian with parameter $K$ in the sense that\footnote{This is one of several equivalent (within constants) definitions of $K$--sub-Gaussianity. See, for example, \cite[Sec.~2.5]{Vershynin2018}.} $\E e^{\abs{w}^2/K^2} \leq 2$.
As noted by \cite{Krahmer2018,Krahmer2020}, certain moments of $\crdvar$ are critical for our ability to do phase retrieval with such measurements:
\begin{itemize}
	\item If $\E \abs{\crdvar}^4 = 1$, or, equivalently, $\abs{\crdvar} = 1$ almost surely, then the standard basis vectors of $\F^d$ will be indistinguishable under these measurements.
	In that case, we must assume that the ground truth $x_*$ is not too ``peaky'' (i.e., that it is \emph{incoherent} with respect to the standard basis).
	\item If $x_*$ is complex, and $\abs{\E \crdvar^2} = 1$ (i.e., almost surely, $\crdvar = s v$ for some fixed $s \in \C$ and a \emph{real} random variable $v$), then $x_*$ and its elementwise complex conjugate $\xbr_*$ will be indistinguishable.
	We must therefore rule out this case.
\end{itemize}
We can plug the lower isometry bounds of \cite{Krahmer2020} into our theory to obtain the following result (see \Cref{sec:subG_proofs} for details):
\begin{theorem}
	\label{thm:subG_final}
	Consider the model \eqref{eq:model_gen} with rank-one $Z_* = x_* x_*^*$ for nonzero $x_* \in \F^d$.
	Suppose $A_i = a_i a_i^*$, where $a_1, \dots, a_n$ are i.i.d.\ random vectors whose entries are i.i.d.\ copies of a random variable $\crdvar$.
	If $\F = \R$ but $w$ is complex, we can take $A_i = \real(a_i a_i^*)$.
	
	There exists a universal constant $\incohp > 0$ such that the following is true.
	Suppose $\E \abs{\crdvar}^2 = 1$, $\crdvar$ is $K$--sub-Gaussian, and at least one of the following two statements holds:
	\begin{enumerate}
		\item $\E \abs{\crdvar}^4 > 1$, or
		\item $\norm{x_*}_\infty \leq \incohp \norm{x_*}$.
	\end{enumerate}
	Furthermore, if $\F = \C$, assume that $\abs{\E \crdvar^2} < 1$.
	Then there exist $c_1, c_2, c_3, c_4 > 0$ depending only on the properties of $\crdvar$ (not on the dimension $d$) such that, if $n \geq c_1 d$, with probability at least $1 - c_2 n^{-2}$,
	for all
	\[
		p \geq c_3\parens*{ 1 + \frac{d \log d}{n} + \frac{\opnorm{\scrA^*(\xi)}}{n \norm{x_*}^2} },
	\]
	every second-order critical point $X$ of \eqref{eq:opt_gen} satisfies
	\[
		\normF{X X^* - x_* x_*^*} \leq \nucnorm{X X^* - x_* x_*^*} \leq c_4 \frac{\opnorm{\scrA^*(\xi)}}{n}.
	\]
\end{theorem}
We see that, even with $n$ of order $d$ (vs.\ $n \gtrsim d \log d$ as required by \Cref{thm:pr_benign_r1}),
we can obtain a benign landscape by choosing $p \approx \log d$.
In terms of computational scaling, this is an improvement over the results of \cite{Krahmer2020}, which only proved exact recovery for a variant of \eqref{eq:phaselift}.
For the noise term, see the discussion after \Cref{thm:gauss_final_nl} below.
See \Cref{sec:relwork_subG} for further relevant literature.

We have so far been unable to extend this analysis approach to larger ground truth ranks without introducing a suboptimal dependence on the rank $r$.
We thus, in addition, consider another (and older) method.

\subsubsection{Dual certificate approach with application to Gaussian measurements}
Our other analysis technique is a \emph{dual certificate} approach similar to that introduced by \cite{Candes2013,Demanet2014} to analyze a variant of \eqref{eq:phaselift} for phase retrieval.
We defer the details to \Cref{sec:dualcerts}.
A deterministic landscape result similar to \Cref{thm:pr_benign_overp} is given as \Cref{thm:dual_thm} in that section.
This result allows for measurement noise.

As an example application of the dual certificate approach, we consider again rank-1 Gaussian measurements.
Although this approach could likely be adapted to the more general sub-Gaussian measurements of \Cref{thm:subG_final} (as is done in the real, $r = 1$ case by \cite{Krahmer2018}), the dual certificate construction and analysis become more complicated, so for brevity we do not explore this further.
For $Z_*$ of rank $r$, we denote its nonzero eigenvalues by $\lambda_1(Z_*) \geq \cdots \geq \lambda_r(Z_*)$.
\begin{theorem}
	\label{thm:gauss_final_nl}
	Consider the model \eqref{eq:model_gen} with fixed rank-$r$ $Z_* \succeq 0$.
	Suppose $A_i = a_i a_i^*$ for i.i.d.\ standard (real or complex) Gaussian vectors $a_1, \dots, a_n$ (if $\F = \R$ but the measurements are complex, we can take $A_i = \real(a_i a_i^*)$).
	
	For universal constants $c_1, c_2, c_3, c_4 > 0$,
	if $n \geq c_1 r d$, then, with probability at least $1 - c_2 n^{-2}$,
	for all optimization ranks
	\[
		p \geq  c_3 \frac{(1 + \frac{d \log d}{n}) \tr Z_* + \frac{1}{n} \opnorm{\scrA^*(\xi)}}{\lambda_r(Z_*)},
	\]
	every second-order critical point $X$ of \eqref{eq:opt_gen} satisfies
	\[
		\normF{X X^* - Z_*} \leq c_4 \sqrt{r} \frac{\opnorm{\scrA^*(\xi)}}{n}.
	\]
\end{theorem}
The dependence of the error bound on the noise $\xi$ and the ground truth rank $r$ is identical to classical results in low-rank matrix sensing and is, in some cases, minimax-optimal.
See, for example, \cite{Candes2011b,Negahban2011,Rohde2011,Cai2015}.
Usually, however, without a hard estimator rank constraint, one must include a low-rank--inducing regularizer (e.g., trace/nuclear norm) to get such optimal dependence on $r$.
The fact that we obtain this without any explicit regularizer illustrates the ``implicit regularization'' of the semidefinite problem structure (see the relevant footnote above).

In the case $r = 1$ and $\xi = 0$, the result matches that of \Cref{thm:subG_final}.
For general $r$, assuming for simplicity that $n \gtrsim d \log d$ and $\xi = 0$, the optimization rank condition becomes $p \gtrsim \frac{\tr Z_*}{\lambda_r(Z_*)}$.
This is satisfied, for example, when $p \gtrsim \kappa r$, where $\kappa = \lambda_1(Z_*)/\lambda_r(Z_*)$.
Thus we see that yet another ``condition number'' appears in a requirement on $p$.
It is not clear whether the dependence on the eigenvalues $Z_*$ is tight;
related works assuming RIP (e.g., \cite{Zhang2025c}) do not have such a dependence,
but relaxing the RIP assumption as we do requires quite different proof techniques.

The most comparable existing theoretical guarantees for a nonconvex approach to \eqref{eq:opt_gen} in this setting are those of Li et al.\ \cite{Li2021a},
who study an initialization plus gradient descent algorithm.
They require $n \gtrsim \kappa^3 r^4 d \log d$ Gaussian measurements,
whereas the statistically optimal sample complexity obtained by \Cref{thm:gauss_final_nl} is $n \approx r d$.
In \Cref{thm:gauss_final_nl}, a dependence on the eigenvalues of $Z_*$ and a factor of $\log d$ only appear in the requirement on the optimization rank $p$,
which affects computational but not statistical complexity.

\subsection{Potential future directions}

\subsubsection*{Computational complexity guarantees}
Much of the phase retrieval literature has carefully considered the problem of the \emph{computational cost} of finding a solution (see, e.g., \cite{Wang2018b,Kim2025}).
We have not attempted to do something similar in the present work, but it should be possible.

A complicating factor for obtaining competitive computational guarantees is that, in the overparametrized case, the objective function is not locally strongly convex (even modulo the trivial action of the orthogonal/unitary group) about a rank-deficient minimizer.
An overview of this issue and further reading is provided by \cite{Zhang2023b}.
That work also proposes a solution via preconditioned gradient descent.
It is likely that their results (e.g., their Cor.~9) could, with some additional calculations of properties of \eqref{eq:opt_gen}, give a computational complexity bound, but we do not pursue this here.

\subsubsection*{Further applications}
We believe that the dual certificate approach of \Cref{sec:dualcerts} can, with additional work, be applied to other (non-Gaussian) measurements the arise in applications.
For example, many papers (e.g., \cite{Candes2015a,Gross2017,Li2022,Li2025,Hu2025}) consider \emph{coded diffraction patterns}, which come from optical imaging;
in particular, the data has the form of optical diffraction images produced with a number of randomly-generated masks.
Certain of these works prove exact recovery results for the semidefinite relaxation PhaseLift via a dual certificate similar to what we use in this paper.
However, for technical reasons, we cannot simply plug their intermediate results into our framework,
so additional work is needed to obtain theoretical landscape guarantees for such a measurement model.

\subsection{Paper outline and additional notation}
The rest of this paper is organized as follows:
\begin{itemize}
	\item \Cref{sec:relwork} gives additional background and related work.
	\item \Cref{sec:basic_calcs} presents the second-order criticality conditions of \eqref{eq:opt_gen} and derives a fundamental deterministic inequality (\Cref{lem:basic_ineq}) that will be foundational for all the results in this paper.
	\item \Cref{sec:AZ_conc} states and proves a probabilistic concentration result (\Cref{lem:AZ_conc}) for the quantity $\scrA^* \scrA(Z_*)$ that appears in \Cref{lem:basic_ineq} and is thus critical to our subsequent results.
	\item \Cref{sec:proof_gausslike} proves the results for Gaussian(-like) measurements (\Cref{thm:pr_benign_r1,thm:pr_benign_overp}) introduced in \Cref{sec:intro_pr_r1,sec:intro_overp}.
	\item \Cref{sec:subG_proofs} gives a proof (based on results from \cite{Krahmer2020}) of the phase retrieval landscape result \Cref{thm:subG_final} for sub-Gaussian measurements given in \Cref{sec:intro_dualcerts}.
	\item \Cref{sec:dualcerts} describes in detail the theoretical machinery of PhaseLift dual certificates (mentioned in \Cref{sec:intro_dualcerts}) and states and proves our main deterministic theoretical result for this analysis (\Cref{thm:dual_thm}).
	We then apply this to the Gaussian measurement ensemble to prove \Cref{thm:gauss_final_nl}, which was given in \Cref{sec:intro_dualcerts}.
\end{itemize}

For convenience, we collect here some (standard) notation that we use throughout the paper.
If $x$ is a vector, we denote its Euclidean ($\ell_2$), $\ell_1$ and $\ell_\infty$ norms by $\norm{x}$, $\norm{x}_1$, and $\norm{x}_\infty$, respectively.
If $X$ is a matrix, we denote its operator, Frobenius (elementwise Euclidean) and nuclear norms by $\opnorm{X}$, $\normF{X}$, and $\nucnorm{X}$, respectively.

Given $A, B \in \herms_d$ (the set of $d \times d$ Hermitian matrices, real or complex according to context), we write $A \preceq B$ (or $B \succeq A$) to mean $B - A \succeq 0$.
We denote by $I_d$ the $d \times d$ identity matrix.
If $X$ is a matrix of rank $r$, we denote its nonzero singular values by $\sigma_1(X) \geq \cdots \sigma_n(X)$.
If $X$ is Hermitian and positive semidefinite, in which case the singular values are the eigenvalues, we may instead write $\lambda_1(X) \geq \cdots \geq \lambda_r(X)$.

\section{Additional background and related work}
\label{sec:relwork}
The phase retrieval literature is vast, and we can only cover a small portion of it that is most relevant to our work.
For further reading, Schechtman et al.\ \cite{Shechtman2015} give an accessible introduction from an optics/image processing point of view.
The recent survey of Dong et al.\ \cite{Dong2023} has a more statistical perspective.
Fannjiang and Strohmer \cite{Fannjiang2020} provide a much longer and more technically detailed overview, including many convex and nonconvex algorithms and their theoretical guarantees.

One parallel line of work which we do not fully cover attempts to show that we can solve nonconvex problems like \eqref{eq:opt_gen}, typically without overparametrization, with a good initialization and local convergence of some algorithm.
This can be done either by proving a ``locally benign'' landscape or by carefully tracking the iterates of a specific algorithm such as gradient descent.
On the one hand, avoiding overparametrization can save some computation.
On the other hand, the statistical guarantees available with such approaches tend to be suboptimal once we go beyond the simplest case of ordinary phase retrieval with Gaussian measurements; see, for example, the disussion of the state-of-the-art results \cite{Li2021a,Peng2024,Gao2021a} below.
See those papers and the older but more comprehensive survey \cite{Chi2019} for further discussion and references.

We also do not attempt to comprehensively survey the literature on nonconvex optimization and benign landscapes for general low-rank matrix sensing without overparametrization.
Outside of phase retrieval (see below) and certain other highly problem-specific results (see, e.g., \cite{Ge2017} for matrix completion and robust principal component analysis),
all the global landscape results we are aware of assume some form of restricted isometry property (RIP).
For state-of-the-art results and further references, see \cite{Bi2022,Zhang2025c,Zhang2025a}.

%In the following subsections, we consider more carefully the literature for phase retrieval and more general semidefinite low-rank matrix sensing.

\subsection{Nonconvex optimization landscapes for phase retrieval}
\label{sec:relwork_landscapes}
For the quartic objective function \eqref{eq:opt_pr_r1},
primarily the Gaussian measurement case has been studied.
The optimal sample-complexity threshold for obtaining a benign landscape is an open question.
Sun et al.\ \cite{Sun2018} showed that $n \gtrsim d \log^3 d$ suffices.
Cai et al.\ \cite{Cai2023} subsequently improved this requirement to $n \gtrsim d \log d$ (our \Cref{thm:pr_benign_r1}, part~\ref{lb:benign_gauss_p2} recovers this result).
Sarao Mannelli et al.\ \cite{SaraoMannelli2020} provide numerical evidence and heuristic (statistical physics) arguments that the landscape indeed becomes benign when $n/d$ passes a constant threshold.
However, Liu et al.\ \cite{Liu2024} study in detail the landscape when $d$ is large and $d \lesssim n \ll d \log d$ and show that local convexity near the global optimum $x_*$ (which is a key part of the arguments of \cite{Sun2018,Cai2023}) breaks down in this regime.
Other works have considered different objective functions.
Davis et al.\ \cite{Davis2020} study the nonsmooth variant of \eqref{eq:opt_pr_r1} $\min_x~\sum_i \abs{y_i - \abs{\ip{a_i}{x}}^2 }$.
They study the locations of critical points but do not obtain a global benign landscape result.
% (they do, however, show that an initialization and local subgradient descent algorithm will recover the ground truth with the optimal $n \approx d$ sample complexity).
The recent series of papers \cite{Li2020a,Cai2022,Cai2021,Cai2022a} considers a variety of loss functions which combine features of \eqref{eq:opt_pr_r1} with truncation and/or features of the nonsmooth amplitude-based loss $\sum_i (\sqrt{y_i} - \abs{\ip{a_i}{x}})^2$.
In each case, they show that, with $n \gtrsim d$ Gaussian measurements,
the nonconvex landscape is benign in the sense that every second-order critical point gives exact recovery of the ground truth.

The literature on more general nonconvex optimization formulations and algorithms for phase retrieval is vast, and we do not attempt to cover it here.
Most existing theoretical results consider initialization and local convergence of iterative algorithms.
%Objective functions include the ``intensity'' formulations like \eqref{eq:opt_pr_r1} which operate on $y_i - \abs{\ip{a_i}{x}}^2 \approx 0$,
%``amplitude'' formulations which operate on $\sqrt{y_i} - \abs{\ip{a_i}{x}}$,
%and various truncated versions of these.
See the above-mentioned surveys and the recent papers \cite{Li2025,Peng2024,Kim2025} for further background and references.

\subsection{Phase retrieval with sub-Gaussian measurements}
\label{sec:relwork_subG}
For the case of phase retrieval with general sub-Gaussian measurements (like in our \Cref{thm:subG_final}),
Eldar and Mendelson \cite{Eldar2014}, considering only the real case, first showed a universal lower (``stability'') bound on (in our notation) $\norm{\scrA(u u^* - v v^*)}_1$
over $u,v \in \R^d$ (or subsets thereof).
Although their analysis framework is quite general,
their concrete examples assume a ``small-ball'' condition on the $a_i$'s that rules out, for example, measurement vectors composed of i.i.d.\ symmetric Bernoulli (zero-mean $\pm 1$-valued) random variables (hence this is qualitatively similar to the fourth-moment assumption $\E \abs{\crdvar}^4 > 1$ of \Cref{thm:subG_final}).

Krahmer and Liu \cite{Krahmer2018} build on that analysis framework and show that we can relax the small-ball (or moment) assumption if we assume that the ground truth vector is not too ``peaky''; this is the assumption $\norm{x_*}_\infty \leq \incohp \norm{x_*}$ of \Cref{thm:subG_final}.
They furthermore show, via a dual certificate approach similar to \cite{Candes2013,Demanet2014}, that, under similar assumptions as our \Cref{thm:subG_final}, a variant of \eqref{eq:phaselift} gives exact recovery.
Krahmer and Stöger \cite{Krahmer2020} extend this to the complex case (albeit without using dual certificates).

Independently, Gao et al.\ \cite{Gao2021a}, under measurement moment assumptions similar to those of \Cref{thm:subG_final},
showed that a spectral initialization plus gradient descent algorithm gives exact recovery when $n \gtrsim d \log^2 d$.

Recently, Peng et al.\ \cite{Peng2024}, with a similar setup as \cite{Krahmer2020} (and thus \Cref{thm:subG_final}), use an intricate leave-one-out analysis
to show that spectral initialization plus gradient descent (with much larger step size than the result of \cite{Gao2021a} allows) gives exact recovery.
Their guarantees require $n \gtrsim d \log^3 d$ measurements.
They comment that, before their work, there was no non-convex algorithm theoretically guaranteed to solve phase retrieval under such assumptions (e.g., symmetric Bernoulli measurements).
Our \Cref{thm:subG_final} gives another nonconvex approach with improved sample complexity via a benign landscape of the least-squares problem \eqref{eq:opt_gen}.

\subsection{Semidefinite low-rank matrix sensing (generalized phase retrieval)}
\label{sec:relwork_gpr}
The more general semidefinite low-rank matrix sensing problem we present in \Cref{sec:intro},
that is, recovery of a matrix $Z_* \succeq 0$ from measurements of the form $\ip{A_i}{Z_*}$ for positive semidefinite (PSD) measurement matrices $A_i \succeq 0$,
is sometimes called \emph{generalized phase retrieval}.
However, this term is not entirely well defined in the literature.
For example, it is used by \cite{Wang2019a} (and certain follow-up works) to denote a variety of problems, including quite general linear matrix sensing.
However, they primarily use this term to mean recovery of a \emph{vector} $x_*$ from quadratic measurements of the form $\ip{A_i}{x_* x_*^*}$ for general (not necessarily PSD) $A_i \in \herms_d$.
%They primarily consider the existence of minimal measurement ensembles $\{A_i\}_i$.
In this section, we only consider cases where both $Z_*$ and the $A_i$'s are PSD.

One special case of semidefinite low-rank matrix sensing is the multidimensional scaling or sensor network localization problem. The work \cite{Criscitiello2025preprint},
written in parallel with the present paper, studies the nonconvex landscape of such problems in detail; certain of the results in that work are special cases of those in the present paper.

Chi and Lu \cite{Chi2016} propose and study numerically an iterative (Kaczmarz) algorithm for recovery of a low-rank PSD matrix from rank-1 PSD measurements. They do not provide theoretical guarantees; existing theoretical analyses of similar algorithms (e.g., in \cite{Tan2019}) only consider the ordinary phase retrieval case $r = 1$.
The best statistical guarantees we are aware of for a nonconvex approach to this general problem (also with rank-1 measurements) are those of Li et al.\ \cite{Li2021a},
who prove convergence of an initialization plus gradient descent algorithm by tracking the iterates.
However, their results are statistically suboptimal in terms of the required sample size; see the discussion after \Cref{thm:gauss_final_nl} in \Cref{sec:intro_dualcerts}.

For convex approaches to the same problem, Chen et al.\ \cite{Chen2015d} analyse a trace-regularized variant of \eqref{eq:phaselift}, though they note that their techniques could extend beyond the case $Z_* \succeq 0$ to recovery of general Hermitian matrices.
Indeed, Kueng et al.\ \cite{Kueng2017} later do exactly this (with some additional extensions).
Both works show that, if $r = \rank(Z_*)$, then $n \gtrsim rd$ Gaussian measurements suffice to ensure recovery with semidefinite programming.
Their analysis depends on the nuclear norm penalty and does not take advantage of PSD structure as we (and, for example, \cite{Krahmer2020,Candes2013,Demanet2014}) do.

Relatedly, Balan and Dock \cite{Balan2022} study loss functions of the form \eqref{eq:opt_gen} as well as
``amplitude''--based loss functions of the form
\[
	\sum_i (\ip{A_i}{X X^*}^{1/2} - \ip{A_i}{Z_*}^{1/2})^2
\]
for general PSD matrices $A_i$.
They focus on explicit calculation of upper and lower isometry constants of these loss functions with respect to certain natural metrics ($\lowergaussF$ from our \Cref{thm:pr_benign_overp} is one example of such a constant).

\subsection{Overparametrization and condition numbers in low-rank matrix optimization}
\label{sec:relwork_overp}
We have seen in \Cref{sec:intro_overp} that we can view overparametrization as a way to overcome a poor condition number of the measurement operator $\scrA$.
As discussed there, this is particularly related to recent work showing benign landscape under a restricted isometry property \cite{Ma2023b,Zhang2025c,Zhang2025a}.
More broadly, overparametrization can be a useful tool to solve general\footnote{The quadratic-cost program \eqref{eq:phaselift} as well as the many variants in the literature can be put in this form, though most works do not do this.} SDPs with linear objective and constraints of the form
\begin{equation}
	\label{eq:SDP_gen}
	\min_{Z \succeq 0}~\ip{C}{Z} \st \scrAbr(Z) = y,
\end{equation}
where, for some dimensions $d', n'$, $C \in \herms_{d'}$, $y \in \R^{n'}$, and $\scrAbr \colon \herms_{d'} \to \R^{n'}$ is linear.

Parametrizing $Z$ by a Burer-Monteiro factorization of the form $X X^*$ for $X \in \F^{n' \times p}$, the resulting nonlinear constraint $\scrAbr(X X^*) = y$ becomes, under certain conditions, a Riemannian manifold constraint \cite{Boumal2019}.
Under these conditions, it is known that the problem \eqref{eq:SDP_gen} always has a solution of rank $\approx \sqrt{n'}$,
and, indeed, if the optimization rank parameter $p$ is chosen to be at least this rank bound, then, for \emph{generic} cost matrices $C$, the optimization landscape is benign, though pathological cases exist where this fails.
See \cite{Boumal2019,OCarroll2022} for relevant results and further reading.

However, for certain problems, we can choose $p$ much smaller than $\sqrt{n'}$.
In addition to the matrix sensing problems we have already discussed,
this phenomenon is well studied for \emph{synchronization} problems.
For certain instances, the optimization landscape is again tied to a condition number (that of a dual certificate matrix to \eqref{eq:SDP_gen}),
and overparametrizing the problem (i.e., choosing the rank parameter $p$ of the Burer-Monteiro factorization to be larger than the rank of the global optimum of \eqref{eq:SDP_gen}) can compensate when the condition number is too large \cite{Abdalla2026,Ling2025,Endor2024,McRae2025preprinta}.

\section{Criticality conditions and basic consequences}
\label{sec:basic_calcs}
All of our theoretical guarantees concern \emph{second-order critical points} of the smooth nonconvex problem \eqref{eq:opt_gen}.
In the real case, $X$ is a second-order critical point if, at $X$, the gradient of the objective function $f_p$ is zero and the Hessian quadratic form is positive semidefinite,
that is,
\begin{equation}
	\label{eq:socp}
	\begin{aligned}
		\nabla f_p(X) &= 0, \quad \text{and}\\
		\nabla^2 f_p(X)[\Xdt, \Xdt] &\geq 0 \quad \text{for all}\quad \Xdt \in \R^{d \times p}.
	\end{aligned}
\end{equation}
In the complex case ($\F = \C$), the meaning is the same, but we must consider \eqref{eq:opt_gen} to be an optimization problem over the real and imaginary parts of the complex variable $X$:
that is, if $X = U + i V$ for $U, V \in \R^{d \times p}$,
we calculate the gradient and Hessian in the variable $(U, V)$.
We make this explicit in our calculations below.

The main result of this section is the following lemma, which is the foundation for every subsequent landscape result in this paper.
A more specialized version of this result appears in the parallel work \cite{Criscitiello2025preprint}.
\begin{lemma}
	\label{lem:basic_ineq}
	Consider \eqref{eq:opt_gen} under the measurement model \eqref{eq:model_gen}.
	If $Z_* \succeq 0$ has rank $r \geq 1$,
	let $X_* \in \F^{d \times r}$ be such that $Z_* = X_* X_*^*$.
	
	Let $X \in \F^{d \times p}$ be a second-order critical point of \eqref{eq:opt_gen}.
	For any matrix $R \in \F^{p \times r}$, we have
	\begin{align*}
		&\norm{\scrA(X X^* - Z_*)}^2  \\
		&\quad\leq \ip{\xi}{\scrA(X X^* - Z_*)} \\
		&\quad\qquad + \frac{2}{p+2} \ip{y}{ \scrA((X_* - X R)(X_* - X R)^*) } \\
		&\quad\leq \ip{\xi}{\scrA(X X^* - Z_*)} + \frac{2 \opnorm{\scrA^*(y)}}{p + 2} \normF{X_* - X R}^2.
	\end{align*}
\end{lemma}
One potential benefit of overparametrization is immediately clear;
the larger $p$, the smaller the last term in the above inequality will be.
\begin{proof}
The second inequality of the result follows from
\begin{align*}
	&\ip{\scrA^*(y)}{(X_* - X R)(X_* - X R)^*} \\
	&\qquad\leq \opnorm{\scrA^*(y)} \nucnorm{(X_* - X R)(X_* - X R)^*} \\
	&\qquad= \opnorm{\scrA^*(y)} \normF{X_* - X R}^2.
\end{align*}
We now turn to the first inequality.

We first consider the real case $\F = \R$,
and then we extend this to the complex case.
Standard calculations give
\begin{align*}
	\nabla f_p(X) &= 4 \scrA^*(A(X X^T) - y) X \\
	&= 4 \scrA^* (\scrA(X X^T - Z_*) - \xi ) X,
\end{align*}
and
\begin{align*}
	\nabla^2 f_p(X)[\Xdt, \Xdt] &= 4 \ip{\scrA^* (\scrA(X X^T) - y) }{\Xdt \Xdt^T} \\
	&\qquad + 2 \norm{\scrA(X \Xdt^T + \Xdt X^T)}^2 \\
	&= 4 \ip{\scrA(X X^T - Z_*) - \xi }{\scrA(\Xdt \Xdt^T)} \\
	&\qquad + 2 \norm{\scrA(X \Xdt^T + \Xdt X^T)}^2.
\end{align*}
We will drop a factor of $4$ from now on, as it has no effect on the criticality conditions \eqref{eq:socp}.

In the case of rank-one measurements $A_i = a_i a_i^T$ and $p = 1$,
we have the convenient identity
\[
	\frac{1}{2} \norm{\scrA(X \Xdt^T + \Xdt X^T)}^2 = 2 \ip{\scrA(X X^T)}{\scrA(\Xdt \Xdt^T)}.
\]
Outside this specific case, this equality does not hold in general, but it becomes an \emph{in}equality that will still be useful.
More precisely, for $A \succeq 0$, we have, by Cauchy-Schwartz,
for any matrices $B,C$ of appropriate size,
\begin{align*}
	\ip{A}{B C^T + C B^T}^2
	&= 4\ip{A^{1/2} B}{A^{1/2} C}^2 \\
	&\leq 4\normF{A^{1/2} B}^2 \normF{A^{1/2} C}^2 \\
	&= 4\ip{A}{B B^T}\ip{A}{C C^T}.
\end{align*}
Therefore, applying this to each $A_i \succeq 0$, we have
\begin{equation}
	\label{eq:psdA_ineq}
	\frac{1}{2} \norm{\scrA(B C^T + C B^T)}^2
	\leq 2 \ip{\scrA(B B^T)}{\scrA(C C^T)}.
\end{equation}
We will consider rank-one $\Xdt$ of the form $\Xdt = u v^T$ for $u \in \R^d, v \in \R^p$.
Plugging this into the Hessian inequality of \eqref{eq:socp} and then applying \eqref{eq:psdA_ineq} gives
\begin{align*}
	0 &\leq \norm{v}^2 \ip{\scrA(X X^T - Z_*) - \xi }{\scrA(u u^T)} \\
	&\qquad + \frac{1}{2} \norm{\scrA(X v u^T + u (X v)^T)}^2 \\
	&\leq  \norm{v}^2 \ip{\scrA(X X^T - Z_*) - \xi }{\scrA(u u^T)} \\
	&\qquad + 2 \ip{\scrA(X v v^T X^T)}{\scrA(u u^T)}.
\end{align*}
Now, for fixed $u$, take $v = v_k$ for each $k = 1,\dots,p$, where $\{v_k\}_k$ is an orthonormal basis for $\R^p$;
adding up the resulting inequalities gives
\begin{align*}
	0 &\leq p \ip{\scrA(X X^T - Z_*) - \xi }{\scrA(u u^T)} \\
	&\qquad + 2 \ip{\scrA(X X^T)}{\scrA(u u^T)} \\
	&= (p+2) \ip{\scrA(X X^T - Z_*) - \xi }{\scrA(u u^T)} \\
	&\qquad + 2 \ip{\underbrace{\scrA(Z_*) + \xi}_{= y}}{\scrA(u u^T)}.
\end{align*}
We will next take, for $R \in \R^{p \times r}$, $u = (X_* - X R)w_\ell$ for $\ell = 1, \dots, r$,
where $\{w_\ell\}_\ell$ is an orthonormal basis for $\R^r$ to obtain, again summing up the resulting inequalities and dividing by $p+2$,
\begin{align*}
	0 &\leq \ip{\scrA(X X^T - Z_*) - \xi }{\scrA((X_* - X R)(X_* - X R)^T)} \\
	&\qquad + \frac{2}{p+2} \ip{y}{\scrA((X_* - X R)(X_* - X R)^T)}.
\end{align*}
Finally, the zero-gradient condition
\[
	\scrA^* (\scrA(X X^T - Z_*) - \xi)X = 0
\]
implies
\begin{align*}
	&\ip{\scrA(X X^T - Z_*) - \xi }{\scrA ((X_* - X R)(X_* - X R)^T)} \\
	&\qquad= \ip{\scrA(X X^T - Z_*) - \xi }{\scrA(X_* X_*^T - X X^T)} \\
	&\qquad= - \norm{\scrA(X X^T - Z_*)}^2 + \ip{\xi}{X X^T - Z_*},
\end{align*}
which we can plug in to the previous inequality to obtain
\begin{align*}
	0 &\leq -\norm{\scrA(X X^T - Z_*)}^2 + \ip{\xi}{\scrA(X X^T - Z_*)} \\
	&\qquad + \frac{2}{p+2} \ip{y}{ \scrA((X_* - X R)(X_* - X R)^T )}.
\end{align*}
This immediately implies the result in the case $\F = \R$.

Now, consider the complex case $\F = \C$.
We rewrite the problem as one over real variables.
Denote by $\symms_{2d}$ the space of symmetric real $2d \times 2d$ matrices.
We use the maps
\begin{align*}
	\herms_d \ni A = B + i C &\mapsto \Atl = \begin{bmatrix*} B & C^T \\ C & B \end{bmatrix*} \in \symms_{2d},\\
	\C^{d \times p} \ni X = U + iV &\mapsto \Xtl = \begin{bmatrix*} U \\ V \end{bmatrix*} \in \R^{2n \times p}.
\end{align*}
Direct calculation confirms that $\ip{A}{X X^*} = \ip{\Atl}{\Xtl \Xtl^T}$.
An immediate consequence is that $A \succeq 0$ implies $\Atl \succeq 0$.
Furthermore, setting $J = \begin{bmatrix*} 0 & -I_d \\ I_d & 0 \end{bmatrix*}$ as the matrix representing multiplication by $i$ (with $J^T = - J$ representing multiplication by $-i$),
we have $J^T \Atl J = \Atl$, which will be useful in the calculations to follow.

The complex problem \eqref{eq:opt_gen} thus reduces to the real optimization problem
\begin{equation}
	\min_{\Xtl \in \R^{2d \times p}}~\norm{\scrAtl(\Xtl \Xtl^T) - y}^2, \label{eq:opt_cplx_real}
\end{equation}
where $\scrAtl \colon \symms_{2d} \to \R^n$ is defined in the same manner as $\scrA$ with the real measurement matrices $\Atl_1, \dots, \Atl_n \in \symms_{2d}$ formed from $A_1, \dots, A_n \in \herms_d$.

The result for the real case then implies that any second-order critical point $\Xtl$ of \eqref{eq:opt_cplx_real} satisfies,
for any $\Rtl \in \R^{p \times r}$,
\begin{align*}
	&\norm{\scrAtl(\Xtl \Xtl^T - \Ztl_*)}^2 \leq \ip{\xi}{\scrAtl(\Xtl \Xtl^T - \Ztl_*)} \\
	&\qquad\qquad + \frac{2}{p+2} \ip{y}{ \scrAtl((\Xtl_* - \Xtl \Rtl)(\Xtl_* - \Xtl \Rtl)^T) },
\end{align*}
where $\Ztl_* \in \symms_{2d}, \Xtl_* \in \R^{2d \times r}$ are defined in the obvious way.
We immediately obtain, by reversing the complex-to-real transformation,
\begin{align*}
	&\norm{\scrA(X X^* - Z_*)}^2 \leq \ip{\xi}{\scrA(X X^* - Z_*)} \\
	&\qquad\qquad + \frac{2}{p+2} \ip{y}{ \scrA((X_* - X \Rtl)(X_* - X \Rtl)^* )}.
\end{align*}
This is not quite what we want, because we had to assume $\Rtl$ was \emph{real}.
We must therefore inspect further the transformed problem's structure and consider how to extend the proof of the real case.
If $R = R_1 + i R_2 \in \C^{p \times r}$ with $R_1, R_2 \in \R^{p \times r}$,
we can replace, in the Hessian inequality calculations, $\Xtl_* - \Xtl \Rtl$ by $\Xtl_* - \Xtl R_1 - J \Xtl R_2$ without any problem.
To use the zero-gradient condition, we need, in addition to the equality
\[
	\scrAtl^* (\scrAtl(\Xtl \Xtl^* - \Ztl_*) - \xi) \Xtl = 0
\]
which is identical to the real case, the equality
\[
	\scrAtl^* (\scrAtl(\Xtl \Xtl^* - \Ztl_*) - \xi) J \Xtl = 0,
\]
which follows from the previous equality by the fact that, for each $i$,
\[
	J^T \Atl_i J = \Atl_i \quad \Longleftrightarrow \quad \Atl_i J = J \Atl_i.
\]
Finally, noting that $X R \longleftrightarrow \Xtl R_1 + J \Xtl R_2$ under the complex-to-real transformation, we indeed obtain the claimed inequality.
\end{proof}

\section{Concentration of \texorpdfstring{$\scrA^*\scrA(Z_*)$}{A*A(Z\_*)} for sub-Gaussian measurements}
\label{sec:AZ_conc}
A key quantity in \Cref{lem:basic_ineq} in the previous section is $\opnorm{\scrA^*(y)} = \opnorm{\scrA^* \scrA(Z_*) + \scrA^*(\xi)}$.
We do not consider noise in detail in this paper (the term $\scrA^*(\xi)$ has been studied in other works on low-rank matrix sensing and phase retrieval; see, e.g., \cite{Candes2011b,Cai2015,McRae2023}),
but we still need to understand the spectral properties of the matrix $\scrA^* \scrA(Z_*)$.
In this section, we provide a concentration result for this matrix when the measurements are sub-Gaussian;
we will use this result in each of our applications.

We say that a zero-mean random vector $a \in \C^d$ is $K$--sub-Gaussian if, for every unit-norm $x \in \C^d$, $\E e^{ \abs{\ip{a}{x}}^2/K^2} \leq 2$.
This is, in particular, true if the entries of $a$ are i.i.d.\ copies of a $K$--sub-Gaussian random variable $\crdvar$ in the sense given in \Cref{sec:intro_subG}.

We explicitly consider the complex case, as it may be the case that the ground truth signal is real, i.e., $\F = \R$, while the measurements are complex. The above definition is still valid when $a$ is real.

The following concentration result is a straightforward extension of \cite[Lem.~21]{Sun2018}. For completeness, we provide a full proof.
\begin{lemma}
	\label{lem:AZ_conc}
	Let $Z_* \succeq 0$ be fixed and of rank $r$.
	Let $a_1, \dots, a_n$ be i.i.d.\ copies of a $K$--sub-Gaussian vector $a \in \C^d$,
	and take $A_i = a_i a_i^*$.
	There exists a universal constant $c > 0$ such that, if $n \geq d$, with probability at least $1 - 3 n^{-2}$,
	\begin{align*}
		&\scrA^* \scrA(Z_*)
		\preceq \E \scrA^* \scrA(Z_*) \\
		&\qquad + c K^2 \parens*{\sqrt{\frac{d + \log n}{n}} +  \frac{(d + \log n) \log n}{n}} (\tr Z_*) I_d.
	\end{align*}
\end{lemma}
\begin{proof}
	We will consider the case $r = 1$ first and then use this to extend to general $r \geq 1$.
	Thus, for now, assume $Z_* = x_* x_*^*$,
	and, furthermore, assume without loss of generality that $\norm{x_*} = 1$.
	
	Within this proof, we use the letters $c$, $c'$, etc.\ to denote universal positive constants that may change from one usage to another.
	
	Note the following facts which follow from the sub-Gaussian assumption on $a$:
	\begin{itemize}
		\item For all unit-norm $x \in \F^d$, $\ip{A}{x x^*} = \abs{\ip{a}{x}}^2 \geq 0$ satisfies $\E e^{\ip{A}{x x^*}/K^2} \leq 2$,
		which implies, for all integers $k \geq 1$,
		\begin{equation}
			\label{eq:Axx_moments}
			\E \ip{A}{x x^*}^k \leq 2 K^{2k} k!.
		\end{equation}
		\item With probability at least $1 - n^{-3}$,
		\begin{equation}
			\label{eq:trunc_event}
			\begin{aligned}
			\max_i~\ip{A_i}{x_* x_*^*} &= \max_i~\abs{\ip{a_i}{x_*}}^2 \\
			& \leq c K^2 \log n \eqqcolon \tau.
			\end{aligned}
		\end{equation}
		See, for example, \cite[Ch.~2]{Vershynin2018} for more details.
	\end{itemize}
	We then, purely for analysis purposes, consider a truncated version $\scrA^* \scrA(x_* x_*^*)$:
	set
	\[
		Y_\tau \coloneqq \sum_{i=1}^n G_i, \quad \text{where} \quad G_i \coloneqq \ip{A_i}{x_* x_*^*} \indicator{\ip{A_i}{x_* x_*^*} \leq \tau} A_i.
	\]
	On the event of \eqref{eq:trunc_event}, we have $Y_\tau = \scrA^* \scrA(x_* x_*^*)$.

	For any unit-norm $x \in \F^d$, the i.i.d.\ and nonnegative random variables $\{ \ip{G_i}{x x^*} \}_i$ satisfy, for $k \geq 2$,
	\begin{align*}
		\E \ip{G_i}{x x^*}^k
		&= \E (\ip{A_i}{x_* x_*^*} \indicator{\ip{A_i}{x_* x_*^*} \leq \tau} \ip{A_i}{x x^*} )^k \\
		&\leq \tau^{k-2} \E [\ip{A_i}{x_* x_*^*}^2 \ip{A_i}{x x^*}^k] \\
		&\leq \tau^{k-2} 2 K^{2(k+2)} (k+2)! \\
		&\leq c' (c K^2 \tau)^{k-2} K^8 k!.
	\end{align*}
	The second inequality comes from Hölder's inequality together with \eqref{eq:Axx_moments}.
	For the last equality, we have absorbed the factor of $(k+2)(k+1)$ into the constant in the exponential and the leading constant.
	
	This implies that the random variable $\ip{Y_\tau}{x x^*} = \sum_{i=1}^n \ip{G_i}{xx^*}$ is $(c K^4 n, c' K^2 \tau)$--sub-expontential in the sense of \cite[Sec.~2.1]{Wainwright2019},
	so, for any $t \geq 0$,
	with probability at least $1 - 2 e^{-t}$,
	\[
		\abs{ \ip{Y_\tau}{x x^*} - \E \ip{Y_\tau}{x x^*} }
		\lesssim K^4 \sqrt{n t} + K^2 \tau t.
	\]
	By a covering argument (see, e.g., \cite[Ch.~4]{Vershynin2018}),
	we then have, with probability at least $1 - 2 e^{-t}$,
	\begin{align*}
		\opnorm{ Y_\tau - \E Y_\tau }
		&\leq c \parens*{K^4 \sqrt{n (d + t)} + K^2 \tau (d + t)} \\
		&\leq c K^4 \parens*{\sqrt{n (d + t)} +  (\log n) (d + t)}.
	\end{align*}
	Taking $t = 3 \log n$ ensures that this holds with probability at least $1 - 2 n^{-3}$.
	Next, note that $\E \scrA^* \scrA(x_* x_*^*) \succeq \E Y_\tau$ (because $\scrA^* \scrA(x_* x_*^*) \succeq Y_\tau$).
	Furthermore, recall that on the event of \eqref{eq:trunc_event} (which occurs with probability at least $1 - n^{-3}$), $\scrA^* \scrA(x_* x_*^*) = Y_\tau$.
	Then, by a union bound, we have, with probability at least $1 - 3 n^{-3}$,
	\begin{align*}
		&\scrA^* \scrA(x_* x_*^*) - \E \scrA^* \scrA(x_* x_*^*)
		\preceq Y_\tau - \E Y_\tau \\
		&\qquad \qquad \preceq c K^4 \parens*{\sqrt{n (d + \log n)} +  (d + \log n) \log n} I_d.
	\end{align*}
	Rescaling by $1/n$ gives the result for $r = 1$.
	Applying it to each term in the eigenvalue decomposition
	\[
		Z_* = \sum_{k=1}^r \lambda_k(Z_*) u_k u_k^*
	\]
	and taking a union bound gives the result with probability at least
	\[ 1 - 3 r n^{-3} \geq 1 - 3 n^{-2}.
	\]
\end{proof}

\section{Analysis with restricted lower isometry}
\label{sec:proof_gausslike}
In this section, we provide proofs of \Cref{thm:pr_benign_r1,thm:pr_benign_overp}.
These results assume no noise ($\xi = 0$) and a rank-one $Z_* = x_* x_*^*$ for some $x_* \in \F^d$

Before continuing to the full proofs,
it is helpful, to see the benefits of overparametrization, to consider how one would prove a simplified version of \Cref{thm:pr_benign_overp}.
In the notation of that result, suppose we ignore the trace term in \eqref{eq:pr_overp_lb} (setting $\lowergaussTR$ to zero) and obtain from \eqref{eq:pr_overp_cond} the condition
\[
	(p + 2) \lowergaussF > 2L.
\]
If $X$ is a second-order critical point of \eqref{eq:opt_gen},
combining the assumed inequalities \eqref{eq:pr_overp_lb} and \eqref{eq:pr_overp_ub} with \Cref{lem:basic_ineq} (with $\xi = 0$) gives, for any $u \in \F^p$,
\begin{align*}
	(p + 2) \lowergaussF \normF{X X^* - x_* x_*^*}^2
	&\leq (p + 2) \norm{\scrA(X X^* - x_* x_*^*)}^2 \\
	&\leq 2 \opnorm{\scrA^* \scrA(x_* x_*^*)}\norm{x_* - X u}^2 \\
	&\leq 2 L \norm{x_*}^2 \norm{x_* - Xu}^2.
\end{align*}
An obvious choice of $u$ is one that minimizes $\norm{x_* - Xu}^2$.
This also means that $x_* - Xu \in \range(X)^\perp$,
which implies
\begin{align*}
	\normF{X X^* - x_* x_*^*}^2
	&= \normF{X (X - x_* u^*)^* - (x_* - Xu) x_*^*}^2 \\
	&= \normF{X (X - x_* u^*)^*}^2 + \normF{(x_* - Xu) x_*^*}^2 \\
	&\geq \norm{x_*}^2 \norm{x_* - Xu}^2.
\end{align*}
We thus obtain the inequality
\[
	(p + 2) \lowergaussF \norm{x_*}^2 \norm{x_* - Xu}^2 \leq 2 L \norm{x_*}^2 \norm{x_* - Xu}^2.
\]
Because we assumed $(p + 2) \lowergaussF > 2 L$, we must have $\norm{x_* - Xu} = 0$.
Tracing back through our inequalities then shows $\normF{X X^* - x_* x_*^*} = 0$,
that is, $X X^* = x_* x_*^*$.

Unfortunately, even for Gaussian measurements, considering the approximate values of $\lowergaussF$ and $L$ suggested by the expectation \eqref{eq:expectation_gauss},
the condition $(p + 2) \lowergaussF > 2 L$ will not be satisfied when $p = 1$;
the above simple analysis is too loose in this case.
We therefore need a more careful analysis that includes the trace term (corresponding to the parameter $\lowergaussTR$ of \Cref{thm:pr_benign_overp}) of \eqref{eq:expectation_gauss}.

We begin with a proof of \Cref{thm:pr_benign_r1}, which only considers optimization rank parameter $p = 1$.
Although part~\ref{lb:benign_gauss_p1} is a direct consequence of the more general result \Cref{thm:pr_benign_overp},
we provide a full proof for pedagogy and motivation, as the additional calculations necessary to incorporate the trace (i.e., $\lowergaussTR$) term of \eqref{eq:pr_overp_lb} are simplest in the case $p = 1$.

First, in order to prove part~\ref{lb:benign_gauss_p2}, we need, along with \Cref{lem:AZ_conc} above, another concentration result for Gaussian measurements.
\begin{lemma}[{\cite[Lemma 22]{Sun2018}}]
	\label{lem:gauss_lower_direct}
	Consider Gaussian measurements of the form $A_i = a_i a_i^*$,
	where $a_1, \dots, a_n$ are i.i.d.\ real or complex standard Gaussian vectors.
	For a function $c_1(\delta) > 0$ only depending on $\delta$ and a universal constant $c_2 > 0$,
	for any $\delta \in (0, 1)$, if $n \geq c_1(\delta) d \log d$, then,
	with probability at least $1 - c_2 n^{-2}$,
	uniformly over $x, z \in \F^d$,
	\begin{align*}
		&\frac{1}{n} \norm{\scrA(x x^* - z z^*)}^2 \\
		&\qquad \geq (1 - \delta) (1 - \delta) [\fourthmoment \normF{x x^* - z z^*}^2 + (\norm{x}^2 - \norm{z}^2)^2 ].
	\end{align*}
\end{lemma}
Recalling \eqref{eq:expectation_gauss},
note that the right-hand size of the inequality is $(1 - \delta) \E \frac{1}{n} \norm{\scrA(x x^* - z z^*)}^2$.
The cited lemma is precisely the above in the complex case with some simplifications in the statement;
the real case holds by the same arguments.

%\begin{lemma}
%	\label{lem:AZ_gauss_r1}
%	Let $x_* \in \F^d$ be fixed,
%	and consider Gaussian measurements of the form $A_i = a_i a_i^*$,
%	where $a_1, \dots, a_n$ are i.i.d.\ standard Gaussian vectors in $\F^d$.
%	For universal constants $c_1, c_2 > 0$,
%	for any $\delta \in (0, 1)$,
%	if $n \geq c_1(\delta^{-2} d + \delta^{-1} d \log d)$, the Gaussian ensemble satisfies, with probability at least $1 - c_2 d^{-2}$,
%	\[
%		\frac{1}{n} \opnorm{\scrA^* \scrA(x_* x_*^*)} \preceq \fourthmoment x_* x_*^* + (1 + \delta) \norm{x_*}^2 I_d
%		\preceq (1 + \fourthmoment +\delta) \norm{x_*}^2 I_d.
%	\]
%\end{lemma}
%This is a variant of \cite[Lemma 7.4]{Candes2015} and \cite[Lemma 21]{Sun2018}.
%A more general version appears as \Cref{lem:AZ_gauss_conc} in \Cref{sec:gauss_body} below.

\begin{proof}[Proof of \Cref{thm:pr_benign_r1}]
	First, we show part~\ref{lb:benign_gauss_p2} of the theorem statement.
	Fix sufficiently small $\delta_U, \delta_L > 0$ for the assumption \eqref{eq:benign_gauss_cond} to hold with $\moment = \fourthmoment$.
	The expectation calculation \eqref{eq:expectation_gauss} gives $\frac{1}{n} \opnorm{\E \scrA^* \scrA(x_* x_*^*)} = (1 + \fourthmoment)\norm{x_*}^2$.
	By \Cref{lem:AZ_conc} (we can take $K \approx 1$), if $n \gtrsim \delta_U^{-2} d + \delta_U^{-1} d \log d$,
	we have, with probability at least $1 - 3 n^{-2}$,
	\begin{align*}
		\frac{1}{n} \opnorm{\E \scrA^* \scrA(x_* x_*^*)}
		&\leq \frac{1}{n} \opnorm{\E \scrA^* \scrA(x_* x_*^*)} + \delta_U \norm{x_*}^2 \\
		&= (1 + \fourthmoment + \delta_U) \norm{x_*}^2.
	\end{align*}
	For the lower isometry assumption, we can directly apply \Cref{lem:gauss_lower_direct};
	if $n \geq c_1(\delta_L) d \log d$, the assumption holds with probability at least $1 - c_2 n^{-2}$.
	Combining the failure probabilities with a union bound gives the result.

	We now turn to proving part~\ref{lb:benign_gauss_p1} of the theorem statement.
	We adopt the cleaner notation of \Cref{thm:pr_benign_overp} and set $\lowergaussF = (1 - \delta_L) \moment$, $\lowergaussTR = 1 - \delta_L$, and $L = 1 + \moment + \delta_U$.
	\Cref{lem:basic_ineq} then implies, for any $s \in \F$,
	\begin{equation}
		\label{eq:rank1_basic}
		\begin{aligned}
		&3 \lowergaussF \normF{x x^* - x_* x_*^*}^2 + 3 \lowergaussTR (\norm{x}^2 - \norm{x_*}^2)^2 \\
		&\qquad\leq 2 L \norm{x_*}^2 \norm{x_* - s x}^2.
		\end{aligned}
	\end{equation}
	The obvious choice of $s$ is the one that minimizes $\norm{x_* - sx}$,
	which, by standard linear algebra calculations, is such that $x_* - s x \perp x$ and (if $x \neq 0$)
	\begin{equation}
		\label{eq:rank1_minc_id}
		\norm{x_* - sx}^2
		= \norm{x_*}^2 - \frac{\abs{\ip{x}{x_*}}^2}{\norm{x}^2}
		= (1 - \rho^2) \norm{x_*}^2,
	\end{equation}
	where
	\[
		\rho^2 \coloneqq \frac{\abs{\ip{x}{x_*}}^2}{\norm{x}^2 \norm{x_*}^2}
	\]
	is the absolute squared correlation between $x$ and $x_*$.
	If $x = 0$, the same holds with $\rho^2 = 0$.
	As $x_* - sx \perp x$,
	we additionally have (denoting by $s^*$ the complex conjugate of $s$)
	\begin{align}
		\normF{x x^* - x_* x_*^*}^2 \nonumber
		&= \normF{x (x - s^* x_*)^* - (x_* - sx) x_*^*}^2 \nonumber \\
		&= \normF{x (x - s^* x_*)^*}^2 + \normF{(x_* - sx) x_*^*}^2 \nonumber \\
		&= \norm{x}^2 \norm{x - s^* x_*}^2 + \norm{x_*}^2 \norm{x_* - sx}^2 \nonumber \\
		&\geq (\norm{x}^4 + \norm{x_*}^4)(1 - \rho^2). \label{eq:rank1_normF_lb}
	\end{align}
	The last inequality uses (cf.\ \eqref{eq:rank1_minc_id})
	\[
		\norm{x - s^* x_*}^2 \geq \min_{s' \in \F}~\norm{x - s' x_*}^2 = (1 - \rho^2)\norm{x}^2.
	\]
	Plugging \eqref{eq:rank1_minc_id} and \eqref{eq:rank1_normF_lb} into \eqref{eq:rank1_basic}, we obtain
	\begin{align*}
		&3 \lowergaussF (\norm{x}^4 + \norm{x_*}^4)(1 - \rho^2) + 3\lowergaussTR (\norm{x}^2 - \norm{x_*}^2)^2 \\
		&\qquad \leq 2 L \norm{x_*}^4(1 - \rho^2).
	\end{align*}
	If $\rho^2 = 1$, then, as $\lowergaussTR = 1 - \delta_L > 0$, we must have $\norm{x}^2 = \norm{x_*}^2$, and we are done.
	If $\rho^2 < 1$,
	then we can divide by $1 - \rho^2$ and obtain the (weaker) inequality
	\[
		3 \lowergaussF (\norm{x}^4 + \norm{x_*}^4) + 3\lowergaussTR (\norm{x}^2 - \norm{x_*}^2)^2
		\leq 2 L \norm{x_*}^4
	\]
	Now assume, without loss of generality, that $\norm{x_*} = 1$,
	and set $t = \norm{x}^2$.
	The above inequality can be rearranged to obtain
	\begin{align*}
		0 &\geq 3(\lowergaussF + \lowergaussTR) t^2 - 6\lowergaussTR t + 3 (\lowergaussF + \lowergaussTR) - 2L \\
		&\geq - \frac{3 \lowergaussTR^2}{\lowergaussF + \lowergaussTR} + 3(\lowergaussF + \lowergaussTR) - 2L,
	\end{align*}
	where the second inequality comes from minimizing the previous expression over $t$ with $t = \frac{\lowergaussTR}{\lowergaussTR + \lowergaussF}$.
	Multiplying by $\lowergaussF + \lowergaussTR$ and rearranging gives
	\[
		3(\lowergaussF^2 + 2 \lowergaussTR \lowergaussF) \leq 2 L (\lowergaussF + \lowergaussTR).
	\]
	Plugging in our values of $\lowergaussF$, $\lowergaussTR$, and $L$ gives
	\[
		3(1 - \delta_L)^2 (\moment^2 + 2 \moment) \leq 2 (1 + \moment + \delta_U) (1 - \delta_L)(1 + \moment).
	\]
	Some algebra gives
	\[
		\moment^2 + 2\moment - 2 \leq 3 (\moment^2 + 2 \moment) \delta_L + 2(\moment + 1) \delta_U,
	\]
	which the condition \eqref{eq:benign_gauss_cond} contradicts.
	This completes the proof.
\end{proof}

With this as a warmup, we now continue to the slightly more complicated general case $p \geq 1$:
\begin{proof}[Proof of \Cref{thm:pr_benign_overp}]
	The inequalities \eqref{eq:pr_overp_lb} and \eqref{eq:pr_overp_ub} and \Cref{lem:basic_ineq} imply, for any $u \in \F^p$,
	\begin{align*}
		&\lowergaussF \normF{X X^* - x_* x_*^*}^2 + \lowergaussTR(\normF{X}^2 - \norm{x_*}^2)^2 \\
		&\qquad \leq \frac{2 L}{p+2} \norm{x_*}^2 \norm{x_* - X u}^2.
	\end{align*}
	We again choose $u$ to minimize $\norm{x_* - X u}^2$.
	Explicitly, we take $u = X^\dagger x_*$, where $X^\dagger$ is the Moore-Penrose pseudoinverse of $X$.
	Again, this ensures that $x_* - Xu \in \range(X)^\perp$, so
	\begin{align*}
		\normF{X X^* - x_* x_*^*}^2
		&= \normF{X(X - x_* u^*)^* - (x_* - Xu)x_*^*}^2 \\
		&= \normF{X(X - x_* u^*)^*}^2 + \norm{x_*}^2 \norm{x_* - X u}^2.
	\end{align*}
	Combined with the previous inequality, we obtain
	\begin{align*}
		&\lowergaussF \normF{X(X - x_* u^*)^*}^2 + \lowergaussTR(\normF{X}^2 - \norm{x_*}^2)^2 \\
	  	&\qquad \leq \parens*{\frac{2L}{p+2} - \lowergaussF} \norm{x_*}^2 \norm{x_* - X u}^2.
	\end{align*}
	We now set
	\[
		\rho^2 \coloneqq \frac{\ip{P_X}{x_* x_*^*}}{\norm{x_*}^2} = \frac{\norm{P_X x_*}^2}{\norm{x_*}^2},
	\]
	where $P_X = X X^\dagger$ is the orthogonal projection matrix onto $\range(X)$.
	Note that in the case $p = 1$, this reduces to the same quantity as in the proof of \Cref{thm:pr_benign_r1} above.
	
	Due to the choice $u = X^\dagger x_*$,
	we have
	\begin{align*}
		\norm{x_* - X u}^2
		&= \norm{x_*}^2 - \norm{X u}^2 \\
		&= \norm{x_*}^2 - \norm{P_X x_*}^2 \\
		&= \norm{x_*}^2(1 - \rho^2).
	\end{align*}
	Furthermore,
	\begin{align*}
		&\normF{X(X - x_* u^*)^*}^2 \\
		&\quad = \normF{X X^* - X u x_*^*}^2 \\
		&\quad = \normF{X X^*}^2 + \normF{X X^\dagger x_* x_*^*}^2 - 2 \ip{X X^*}{X X^\dagger x_* x_*^*} \\
		&\quad = \normF{X X^*}^2 + \norm{x_*}^2 \norm{P_X x_*}^2 - 2 \ip{X X^*}{x_* x_*^*} \\
		&\quad \geq \normF{X X^*}^2 + \norm{x_*}^2 \norm{P_X x_*}^2 - 2 \normF{X X^*} \norm{P_X x_*}^2 \\
		&\quad = \normF{X X^*}^2 + \norm{x_*}^4 \rho^2 - 2 \normF{X X^*} \norm{x_*}^2 \rho^2 \\
		&\quad = (1 - \rho^2) \normF{X X^*}^2 + \rho^2(\normF{X X^*} - \norm{x_*}^2)^2 \\
		&\quad \geq (1 - \rho^2) \normF{X X^*}^2.
	\end{align*}
	We thus obtain
	\begin{align*}
		&\lowergaussF \normF{X X^*}^2 (1 - \rho^2) + \lowergaussTR(\normF{X}^2 - \norm{x_*}^2)^2 \\
		&\qquad \leq \parens*{\frac{2L}{p+2} - \lowergaussF} \norm{x_*}^4 (1 - \rho^2).
	\end{align*}
	If $\rho^2 = 1$, then tracing through our inequalities reveals (as $\lowergaussF > 0$) $\normF{X X^* - x_* x_*^*} = 0$.
	Otherwise, dividing through by $1 - \rho^2$ and noting that $\normF{X}^2 = \tr(X X^*) \leq \sqrt{p} \normF{X X^*}$, we obtain the weaker inequality
	\[
		\frac{\lowergaussF}{p} \normF{X}^4 + \lowergaussTR(\normF{X}^2 - \norm{x_*}^2)^2
		\leq \parens*{\frac{2L}{p+2} - \lowergaussF} \norm{x_*}^4.
	\]
	The rest is similar to the proof of \Cref{thm:pr_benign_r1} above.
	Assume, without loss of generality, that $\norm{x_*} = 1$,
	and set $t = \normF{X}^2$.
	The last inequality can be rewritten as
	\begin{align*}
		0 &\geq \frac{\lowergaussF}{p} t^2 + \lowergaussTR(t - 1)^2
		 - \parens*{\frac{2L}{p+2} - \lowergaussF} \\
		&= \parens*{\frac{\lowergaussF}{p} + \lowergaussTR} t^2 - 2 \lowergaussTR t + \lowergaussF + \lowergaussTR - \frac{2L}{p+2} \\
		&\geq -\frac{\lowergaussTR^2}{\lowergaussTR + \lowergaussF/p} + \lowergaussF + \lowergaussTR - \frac{2L}{p+2} \\
		&= \lowergaussF + \frac{\lowergaussF \lowergaussTR}{p \lowergaussTR + \lowergaussF} - \frac{2L}{p+2}.
	\end{align*}
	The second inequality comes from minimization over $t$ with $t = \frac{\lowergaussTR}{\lowergaussTR + \lowergaussF/p}$.
	The condition \eqref{eq:pr_overp_cond} implies that this last expression is strictly positive, giving a contradiction.
\end{proof}

\section{Sub-Gaussian measurements}
\label{sec:subG_proofs}
In this section, we prove \Cref{thm:subG_final}.
We could prove it via the framework of \Cref{sec:dualcerts};
however, this would require the construction of a suitable dual certificate, which is more complicated for our general sub-Gaussian measurements than it is for the Gaussian measurements we consider in \Cref{thm:gauss_final_nl}.
The paper \cite{Krahmer2018} does this under similar assumptions to \Cref{thm:subG_final} but does not include the complex case.
Instead, we can easily prove \Cref{thm:subG_final} with the help of the following technical lemma of \cite{Krahmer2020}:
\begin{lemma}
	\label{lem:iso_subG}
	Under the conditions of \Cref{thm:subG_final},
	there exist constants $c_1, c_2, c_3, c_4 > 0$ depending only on the properties of $\crdvar$ such that,
	for $n \geq c_1 d$,
	with probability at least $1 - c_2 e^{-c_3 n}$, for all $Z \succeq 0$,
	\[
		\frac{1}{\sqrt{n}} \norm{\scrA(Z - Z_*)}
		\geq \frac{1}{n} \norm{\scrA(Z - Z_*)}_1
		\geq c_4 \nucnorm{Z - Z_*}.
	\]
\end{lemma}
This summarizes several intermediate results of \cite{Krahmer2020} (in particular, their Lemmas 3, 4, and 5).
With this, we can continue to the main proof:
\begin{proof}[Proof of \Cref{thm:subG_final}]
	We will use $c, c'$, etc.\ to denote positive constants, depending only on the properties of $\crdvar$, which may change from one use to another.
	
	Expectation calculations (e.g., \cite[Lem.~9]{Krahmer2020}) give
	\begin{align*}
		\E \frac{1}{n} \scrA^* \scrA(Z_*)
		&= (\tr Z_*) I_d + Z_* + (\abs{\E \crdvar^2}^2) \Zbr_* \\
		&\qquad + (\E \abs{\crdvar}^4 - 2 - \abs{\E \crdvar^2}^2) \ddiag(Z_*) \\
		&\preceq c \norm{x_*}^2 I_d,
	\end{align*}
	where $\Zbr_*$ is the elementwise complex conjugate of $Z_*$,
	and $\ddiag\colon \herms_d \to \herms_d$ extracts the diagonal entries of a matrix.
	Together with \Cref{lem:AZ_conc}, we obtain, with probability at least $1 - c n^{-2}$,
	\begin{align*}
		\frac{1}{n} \opnorm{\scrA^* \scrA(Z_*)}
		&\leq c \parens*{1 +  \frac{d \log n}{n}} \norm{x_*}^2 \\
		&\leq c \parens*{1 +  \frac{d \log d}{n}} \norm{x_*}^2.
	\end{align*}
	The second inequality follows from the observation that, for $n \geq d$, $\frac{d \log n}{n} \lesssim \max\{ 1, \frac{d \log d}{n} \}$.
	
	Next, \Cref{lem:iso_subG} (we relax the probability bound) gives, with probability at least $1 - c n^{-2}$, for all $Z \succeq 0$,
	\[
		\frac{1}{n} \norm{\scrA( Z - Z_* )}^2 \geq c \nucnorm{Z - Z_*}^2.
	\]
	By \Cref{lem:basic_ineq}, we obtain,
	on the union of the above events (with a union bound on the final failure probability),
	for any second-order critical point $X$ and any $u \in \F^p$,
	\begin{align*}
		&c n\nucnorm{X X^* - Z_*}^2 \\
		&\quad\leq \norm{\scrA(X X^* - Z_*)}^2 \\
		&\quad \leq \ip{\xi}{\scrA(X X^* - Z_*)} + \frac{2 \opnorm{\scrA^*(y)}}{p + 2} \norm{x_* - X u}^2 \\
		&\quad \leq \opnorm{\scrA^*(\xi)} \nucnorm{X X^* - Z_*} \\
		&\qquad + \frac{2}{p+2} \parens*{ c'(n + d \log d) \norm{x_*}^2 + \opnorm{\scrA^*(\xi)} } \norm{x_* - X u}^2,
	\end{align*}
	As in the proof of \Cref{thm:pr_benign_overp} in \Cref{sec:proof_gausslike},
	we choose $u$ to minimize $\norm{x_* - X u}$,
	and then
	\[
		\nucnorm{X X^* - Z_*}^2 \geq \normF{X X^* - Z_*}^2 \geq \norm{x_*}^2 \norm{x_* - X u}^2.
	\]
	Then, if
	\[
		p \geq c \parens*{ 1 + \frac{d \log d}{n} + \frac{\opnorm{\scrA^*(\xi)}}{n \norm{x_*}^2} },
	\]
	we obtain
	\[
		\nucnorm{X X^* - Z_*}^2 \leq c \frac{\opnorm{\scrA^*(\xi)}}{n} \nucnorm{X X^* - Z_*},
	\]
	from which the claimed error bound immediately follows.
\end{proof}

\section{PhaseLift dual certificate}
\label{sec:dualcerts}
In this section, we develop our landscape analysis of \eqref{eq:opt_gen} by the method of dual certificates.
We fix the rank-$r$ ground-truth matrix $Z_* = X_* X_*^*$,
where $X_* \in \F^{d \times r}$.
Let $Z_*$ have eigenvalue decomposition $Z_* = U \Lambda U^*$,
where $U \in \F^{d \times r}$ with $U^* U = I_r$,
and $\Lambda$ is an $r \times r$ diagonal matrix with diagonal entries $\lambda_1(Z_*) \geq \cdots \geq \lambda_r(Z_*) > 0$.
Note that, for $k = 1, \dots, r$, $\lambda_k(Z_*) = \sigma_k^2(X_*)$.

We write $\Pmat \coloneqq U U^*$ and $\Ppmat \coloneqq I_d - \Pmat$ as the orthogonal projection matrices onto $\range(Z_*)$ and its orthogonal complement respectively.
We denote by $\T$ the tangent space of rank-$r$ matrices at $Z_*$, given by
\[
\T = \{ U B^* + B U^* : B \in \F^{d \times r} \} \subset \herms_d.
\]
We denote by $\Tp$ its orthogonal complement in $\herms_d$ (with respect to the Frobenius inner product).
The orthogonal projections onto $\T$ and $\Tp$ are respectively given, for $S \in \herms_d$,  by
\begin{align*}
	\PT(S) &= S \Pmat + \Pmat S \Ppmat = \Pmat S + \Ppmat S \Pmat \quad \text{and} \\
	\PTp(S) &= \Ppmat S \Ppmat.
\end{align*}

For a deterministic landscape result, we will make two key assumptions, which resemble those made and, for certain measurement models, proved in papers studying PhaseLift such as \cite{Candes2013,Demanet2014,Candes2015a}.
\begin{assumption}[Dual certificate]
	\label{assump:gen_dual}
	For some $\epsilon \geq 0$, there exists $\lambda \in \R^n$ such that $Y \coloneqq \scrA^*(\lambda)$ satisfies
	\begin{align*}
		\PTp(Y) &\succeq \Ppmat \\
		\normF{\PT(Y)} &\leq \epsilon.
	\end{align*}
\end{assumption}
This is simply a higher-rank analog of the inexact dual certificate introduced in \cite{Candes2013,Demanet2014}.
The quantity $\norm{\lambda}$ will be important in our analysis.

\newcommand{\loweriso}{\mu_{\scrT}}
\newcommand{\upperiso}{L_{\scrT}}
\begin{assumption}[Approximate isometry]
	\label{assump:gen_iso}
	For some $\loweriso, \upperiso > 0$,
	\begin{align*}
		\frac{1}{\sqrt{n}} \norm{\scrA(H)} \geq \loweriso \normF{\PT(H)} - \upperiso \tr(\PTp(H))
	\end{align*}
	for all $H \in \herms_d$ with $\PTp(H) \succeq 0$.
\end{assumption}
The papers \cite{Candes2013,Demanet2014}
instead used the separate assumptions
$\frac{1}{n} \norm{\scrA(H)}_1 \geq \loweriso \normF{H}$ for all $H \in \T$
and	$\frac{1}{n} \norm{\scrA(H)}_1 \leq \upperiso \tr H$ for all $H \succeq 0$.
The combination of these, together with the norm inequality $\norm{\scrA(H)} \geq \frac{1}{\sqrt{n}} \norm{\scrA(H)}_1$, immediately implies \Cref{assump:gen_iso},
but this separation turns out to be suboptimal for our derived results.

We can now state our main deterministic result:
\newcommand{\pthresh}{\tau}
\begin{theorem}
	\label{thm:dual_thm}
	Suppose \Cref{assump:gen_dual,assump:gen_iso} hold with $\loweriso > \upperiso \epsilon$,
	and suppose the rank parameter $p$ in \eqref{eq:opt_gen} satisfies
	\begin{equation}
		\label{eq:noisy_dual_cond}
		p > \pthresh \coloneqq 2 \parens*{\frac{1 + \upperiso \sqrt{n} \norm{\lambda}}{\loweriso - \upperiso \epsilon}}^2 \frac{\opnorm{\scrA^*(y)}}{n \lambda_r(Z_*)} - 2.
	\end{equation}
	Then every second-order critical point $X$ of \eqref{eq:opt_gen} satisfies
	\begin{align*}
		&\normF{X X^* - Z_*} \\
		&\quad \leq \normF{\PT(X X^* - Z_*)} + \tr \PTp(X X^*) \\
		&\quad \leq \frac{p+2}{p - \pthresh} \parens*{ \frac{1 + \epsilon +  \sqrt{n} \norm{\lambda} (\upperiso + \loweriso) }{\loweriso - \upperiso\epsilon} }^2 \sqrt{2r} \frac{\opnorm{\scrA^*(\xi)}}{n}.
	\end{align*}
\end{theorem}
\begin{proof}
	Let $X$ be a second-order critical point of \eqref{eq:opt_gen}.
	Then, for any matrix $R \in \F^{p \times r}$, \Cref{lem:basic_ineq} gives
	\begin{align*}
		\norm{\scrA(X X^* - Z_*)}^2 &\leq \ip{\xi}{\scrA(X X^* - Z_*)} \\
		&\qquad + \frac{2 \opnorm{\scrA^*(y)}}{p + 2} \normF{X_* - X R}^2.
	\end{align*}
	
	Set $H = X X^* - Z_*$.
	We rearrange the previous inequality as
	\begin{equation}
		\begin{aligned}
			&\frac{p - \pthresh}{p+2} \norm{\scrA(H)}^2 
			\leq \ip{\xi}{\scrA(H)} \\
			&\quad + \frac{ 2 \opnorm{\scrA^*(y)} \normF{X_* - X R}^2 - (\pthresh+2) \norm{\scrA(H)}^2 }{p+2}.
		\end{aligned} \label{eq:dual_proof_decomp}
	\end{equation}
	We first consider the second term on the right-hand side of \eqref{eq:dual_proof_decomp}, showing that it cannot be positive.
	We need to lower bound $\norm{\scrA(H)}$.
	We will denote, for brevity,
	\[
		\HT = \PT(H) \quad \text{and} \quad \HTp = \PTp(H) = \Ppmat X X^* \Ppmat \succeq 0.
	\]
	From Cauchy-Schwartz and \Cref{assump:gen_dual}, we have
	\begin{align*}
		\norm{\lambda} \norm{\scrA(H)}
		&\geq \ip{Y}{H} \\
		&= \ip{\PTp(Y)}{\HTp} + \ip{\PT(Y)}{\HT} \\
		&\geq \tr \HTp - \epsilon \normF{\HT}. %\label{eq:AH_ineq_dual}.
	\end{align*}
	We can add this (scaled by $\upperiso$) to the inequality from \Cref{assump:gen_iso} to obtain
	\begin{align*}
		\parens*{\frac{1}{\sqrt{n}} + \upperiso \norm{\lambda}} \norm{\scrA(H)}
		&\geq (\loweriso - \upperiso \epsilon) \normF{\HT},
	\end{align*}
	which implies
	\begin{equation}
		\label{eq:emp_lb_noiseless}
		\norm{\scrA(H)}^2 \geq n \parens*{\frac{\loweriso - \upperiso \epsilon}{1 + \upperiso \sqrt{n} \norm{\lambda}}}^2 \normF{\HT}^2.
	\end{equation}
	Now, choose $R \in \F^{p \times r}$ to minimize $\normF{X_* - X R}$.
	Optimality of $R$ is equivalent to each row of $X_* - X R$ being orthogonal to $\range(X)$,
	that is, $X_* - X R = \PXp X_*$,
	where $\PXp$ is the orthogonal projection matrix onto $\range(X)^\perp \subseteq \F^d$.
	We then have
	\begin{align}
		\normF{\HT}
		&\geq \normF{\Pmat H} \nonumber \\
		&\geq \normF{\Pmat H \PXp} \nonumber \\
		&= \normF{X_* X_*^* \PXp} \nonumber \\
		&= \normF{X_* (X_* - X R)^*} \nonumber \\
		&\geq \sigma_r(X_*) \normF{X_* - X R}. \label{eq:normHT_lb}
	\end{align}
	
	Combining \eqref{eq:emp_lb_noiseless} and \eqref{eq:normHT_lb} and recalling that $\lambda_r(Z_*) = \sigma_r^2(X_*)$, we obtain
	\begin{align*}
		\norm{\scrA(H)}^2
		&\geq n \parens*{\frac{\loweriso - \upperiso \epsilon}{1 + \upperiso \sqrt{n} \norm{\lambda}}}^2 \lambda_r(Z_*) \normF{X_* - XR}^2 \\
		&= \frac{2}{\pthresh + 2} \opnorm{\scrA^*(y)} \normF{X_* - X R}^2.
	\end{align*}
	Using this to simplify \eqref{eq:dual_proof_decomp}, we obtain
	\begin{equation}
		\label{eq:det_noisy_ineq}
		\frac{p - \pthresh}{p+2} \norm{\scrA(H)}^2
		\leq \ip{\xi}{\scrA(H)}
		\leq \opnorm{\scrA^*(\xi)} \nucnorm{H}.
	\end{equation}
	
	The previous inequalities $\norm{\lambda} \norm{\scrA(H)} \geq \tr \HTp - \epsilon \normF{\HT}$ and $\frac{1}{\sqrt{n}} \norm{\scrA(H)} \geq \loweriso \normF{\HT} - \upperiso \tr \HTp$,
	combined in different proportions than before, give
	\begin{align*}
		&\parens*{ (\upperiso+\loweriso) \norm{\lambda} + \frac{1 + \epsilon}{\sqrt{n}} } \norm{\scrA(H)} \\
		&\qquad \geq (\loweriso - \upperiso\epsilon) ( \normF{\HT} + \tr \HTp ),
	\end{align*}
	which implies
	\begin{equation}
		\label{eq:emp_lb_noisy}
		\norm{\scrA(H)}^2 \geq n \mutl ( \normF{\HT} + \tr \HTp)^2,
	\end{equation}
	where $\mutl \coloneqq \parens*{ \frac{\loweriso - \upperiso\epsilon}{1 + \epsilon + \sqrt{n} \norm{\lambda} (\upperiso + \loweriso) } }^2$.
	On the other hand,
	\[
		\nucnorm{H} \leq \nucnorm{\HT} + \nucnorm{\HTp} \leq \sqrt{2r} \normF{\HT} + \tr \HTp.
	\]
	Combining this with \eqref{eq:det_noisy_ineq} and \eqref{eq:emp_lb_noisy},
	we obtain
	\begin{align*}
		&\frac{p - \pthresh}{p+2} n \mutl ( \normF{\HT} + \tr \HTp)^2 \\
		&\qquad \leq \opnorm{\scrA^*(\xi)} (\sqrt{2r} \normF{\HT} + \tr \HTp),
	\end{align*}
	from which the result easily follows.
\end{proof}

\subsection{Application: Gaussian measurements}
\label{sec:gauss_body}
In this section, we show how \Cref{thm:dual_thm} can be applied with the Gaussian measurement model to prove \Cref{thm:gauss_final_nl}.
In this model, the measurement matrices are $A_i = a_i a_i^*$ for i.i.d.\ standard real or complex Gaussian vectors $a_1, \dots, a_n$.
One might reasonably ask why we could not prove \Cref{thm:gauss_final_nl} similarly to the proof of \Cref{thm:subG_final} in \Cref{sec:subG_proofs}; we could then avoid all the complicated machinery of \Cref{thm:dual_thm}.
The reason is that, even with Gaussian measurements, there is no straightforward extension of \Cref{lem:iso_subG} to general $r > 1$ that does not have a problematic dependence on $r$ (in a landscape result, this would require $p$ to scale superlinearly in $r$).
Hence another method is needed.

In this section, we will use liberally the notation $a \lesssim b$ (or $b \gtrsim a$) to mean that $a \leq C b$ for some unspecified but universal constant $C > 0$.
Similarly, the constant $c$ that appears in the probability estimates will not depend on the problem parameters but can change from one usage to another.

We need several supporting lemmas showing that the conditions of \Cref{thm:dual_thm} are satisfied with high probability.
\begin{lemma}
	\label{lem:gauss_dualcert}
	For fixed rank-$r$ $Z_* \succeq 0$,
	if $n \gtrsim r d$,
	then, with probability at least $1 - c n^{-2}$,
	the Gaussian measurement ensemble satisfies
	\Cref{assump:gen_dual} with
	\begin{align*}
		\norm{\lambda} &\lesssim \sqrt{\frac{r}{n}} \quad \text{and} \\
		\epsilon &\lesssim \sqrt{\frac{r^2 (d + \log n)}{n}}.
	\end{align*}
\end{lemma}
This is a straightforward generalization of \cite[Lemma 2.3]{Candes2013} and \cite[Theorem 1]{Demanet2014}, which only consider $r = 1$ and do not bound $\norm{\lambda}$.
We provide a proof below in \Cref{sec:gauss_lemmas}.

\begin{lemma}
	\label{lem:gauss_iso}
	For fixed rank-$r$ $Z_*$,
	if $n \gtrsim r d$, with probability at least $1 - n^{-2}$,
	the Gaussian measurement ensemble satisfies \Cref{assump:gen_iso} with
	\begin{align*}
		\loweriso &\gtrsim 1 \\
		\upperiso &\lesssim \sqrt{\frac{d + \log n}{n}}.
	\end{align*}
\end{lemma}
We provide a proof below in \Cref{sec:gauss_lemmas}.
The methods of \cite{Candes2013,Demanet2014} would provide a similar result with $\upperiso \approx 1$, but, considering the fact that the bounds on $\epsilon$ and $\norm{\lambda}$ in \Cref{lem:gauss_dualcert} increase with $r$, this is suboptimal for larger $r$.

With these tools, we can proceed to the main proof:
\begin{proof}[Proof of \Cref{thm:gauss_final_nl}]
	The failure probabilities of the supporting lemmas are of order $n^{-2}$, so, taking a union bound, the final result has failure probability of the same order.
	
	By the expectation calculation \eqref{eq:expectation_gauss}, \Cref{lem:AZ_conc}, and the fact that $n \gtrsim d$, we have, similarly to the proof of \Cref{thm:subG_final} in \Cref{sec:subG_proofs},
	\[
		\frac{1}{n} \opnorm{\scrA^* \scrA(Z_*)}
		\lesssim \parens*{ 1 + \frac{d \log d}{n}} \tr Z_*.
	\]
	\Cref{lem:gauss_dualcert,lem:gauss_iso} imply that \Cref{assump:gen_dual,assump:gen_iso} hold with
	\begin{gather*}
		\norm{\lambda} \lesssim \sqrt{\frac{r}{n}}, \quad \epsilon \lesssim \sqrt{\frac{r^2 (d + \log n)}{n}}, \\ \loweriso \gtrsim 1, \quad \text{and} \quad \upperiso \lesssim \sqrt{\frac{d + \log n}{n}},
	\end{gather*}
	so
	\[
		\upperiso \sqrt{n}\norm{\lambda} \lesssim \sqrt{\frac{rd}{n}}, \quad \text{and} \quad \upperiso\epsilon \lesssim \frac{r(d + \log n)}{n}.
	\]
	With $n \gtrsim rd$ with large enough constant,
	we will have $\upperiso \sqrt{n} \norm{\lambda} \leq 1/2$ and $\upperiso \epsilon \leq \loweriso/2$,
	so the quantity $\tau$ from \Cref{thm:dual_thm} can be upper bounded as
	\begin{align*}
		\tau
		&\leq \frac{18}{\loweriso^2} \frac{\opnorm{\scrA^* \scrA(Z_*)} + \opnorm{\scrA^*(\xi)}}{n \lambda_r(Z_*)} - 2 \\
		&\lesssim  \frac{(1 + \frac{d \log d}{n}) \tr Z_* + \frac{1}{n}\opnorm{\scrA^*(\xi)}}{\lambda_r(Z_*)}.
	\end{align*}
	We then apply \Cref{thm:dual_thm}.
\end{proof}

\subsection{Proofs of auxiliary lemmas}
\label{sec:gauss_lemmas}
In this section we provide proofs of \Cref{lem:gauss_dualcert,lem:gauss_iso}, which we used to prove \Cref{thm:gauss_final_nl}.

\newcommand{\fourthmomenttrunc}{\moment_4^\gamma}
\newcommand{\secondmomenttrunc}{\moment_2^\gamma}
\begin{proof}[Proof of \Cref{lem:gauss_dualcert}]
	If the matrix $U$ has columns $u_1, \dots, u_r$ (these are the nontrivial eigenvectors of $Z_*$),
	set $E_k = u_k u_k^*$.
	
	We will set, for constants $\alpha, \beta, \gamma > 0$ that we will tune,
	\[
		\lambda_i = \frac{1}{n} \parens*{ \alpha - \beta \sum_{k = 1}^r \ip{A_i}{E_k} \indicator{ \ip{A_i}{E_k} \leq \gamma } }.
	\]
	By construction and the properties of Gaussian random vectors,
	note that, for each $i$, the $r$ random variables $\{ \ip{A_i}{E_k} \}_{k}$ are i.i.d.\ random variables with the distribution of $\abs{z}^2$, where $z$ is a standard normal random variable (real or complex, as appropriate).
	
	Then, we can calculate, by similar methods as for \eqref{eq:expectation_gauss},
	\begin{align*}
		\E Y
		&= n \E \lambda_1 A_1 \\
		&= \alpha I_d - \beta \parens{ \fourthmomenttrunc \Pmat + r \secondmomenttrunc I_d } \\
		&= \parens{ \alpha - \beta \fourthmomenttrunc - \beta r \secondmomenttrunc } \Pmat + (\alpha - \beta r \secondmomenttrunc) \Ppmat,
	\end{align*}
	where
	\begin{align*}
		\secondmomenttrunc &\coloneqq \E[ \abs{z}^2 \indicator{\abs{z}^2 \leq \gamma} ], \qquad \text{and} \\ \fourthmomenttrunc &\coloneqq \E[ \abs{z}^4 \indicator{\abs{z}^2 \leq \gamma} ] - \secondmomenttrunc.
	\end{align*}
	Setting $\alpha = \parens{\fourthmomenttrunc + r \secondmomenttrunc }\beta$,
	we obtain
	\[
		\E Y = \beta \fourthmomenttrunc \Ppmat.
	\]
	We can then set $\gamma$ to be a moderate constant (say, 10) so that $\fourthmomenttrunc \gtrsim 1$ and then set $\beta = (\fourthmomenttrunc)^{-1}$ to obtain $\E Y = \Ppmat$.
	
	It will be useful to bound certain moments of the i.i.d.\ random variables $\lambda_1, \dots, \lambda_n$.
	By construction, $\E \lambda_1 = \frac{1}{n}$.
	Note, furthermore, that we can write
	\[
		\lambda_1 = \E \lambda_1 + \frac{\beta}{n} \sum_{k=1}^r \underbrace{(\secondmomenttrunc - \ip{A_1}{E_k} \indicator{ \ip{A_1}{E_k} \leq \gamma } )}_{\eqqcolon \varepsilon_k}.
	\]
	Recall from above that, because $a_1$ is Gaussian, $\varepsilon_1, \dots, \varepsilon_r$ are i.i.d.\ zero-mean random variables.
	Furthermore, $\E \varepsilon_1^2 \lesssim 1$ and $\E \varepsilon_1^4 \lesssim 1$.
	We can therefore estimate (noting that we have chosen $\beta \lesssim 1$)
	\begin{align*}
		\E \lambda_1^2
		&= (\E \lambda_1)^2 + \frac{\beta^2}{n^2} \sum_{k=1}^r \E \varepsilon_k^2 \\
		&\lesssim \frac{r}{n^2},
	\end{align*}
	and
	\begin{align*}
		\E \lambda_1^4
		&\lesssim (\E \lambda_1)^4 + \frac{\beta^4}{n^4} \E \parens*{ \sum_{k=1}^r \varepsilon_k }^4 \\
		&= \frac{1}{n^4} + \frac{\beta^4}{n^4} \sum_{k,\ell = 1}^r \E[ \varepsilon_k^2 \varepsilon_\ell^2 ] \\
		&\lesssim \frac{r^2}{n^4}.
	\end{align*}
	
	We now bound $\norm{\lambda}$.
	Note that
	\[
		\E \norm{\lambda}^2 = n \E \lambda_1^2 \lesssim \frac{r}{n},
	\]
	so, by Jensen's inequality $\E \norm{\lambda} \lesssim \sqrt{\frac{r}{n}}$.
	Noting furthermore that, by construction, each $\abs{\lambda_i} \lesssim \frac{r}{n}$ almost surely,
	a standard concentration inequality for Lipschitz functions of independent and bounded random variables (e.g., \cite[Thm.~6.10]{Boucheron2013}) gives,
	for $t \geq 0$, with probability at least $1 - e^{-t^2/2}$,
	\[
		\norm{\lambda} - \E \norm{\lambda} \lesssim \frac{r}{n} t.
	\]
	Then, choosing $t = 2 \sqrt{\log n}$,
	we obtain, with probability at least $1 - n^{-2}$,
	\[
		\norm{\lambda} \lesssim \sqrt{\frac{r}{n}} + \frac{r}{n} \sqrt{\log n}
		\lesssim \sqrt{\frac{r}{n}},
	\]
	where the last inequality uses the fact that, for $n \gtrsim r d$,
	\[
		\frac{r \log n}{n} \leq \max\braces*{\frac{rd}{n}, \frac{rd}{e^d} } \lesssim 1.
	\]
	
	We now turn to the concentration of $Y = \scrA^*(\lambda)$ about its mean.
	We use a similar approach as in the proof of \Cref{lem:AZ_conc}.
	Again, $c$, $c'$, etc.\ denote universal positive constants which may change from one appearance to another.
	
	For fixed unit-norm $x \in \F^d$,
	\[
		\ip{Y}{x x^*} = \sum_{i=1}^n \lambda_i \abs{\ip{a_i}{x}}^2.
	\]
	Noting that $\abs{\lambda_i} \leq c \frac{r}{n}$ almost surely, we can bound the moments of each (i.i.d.) term in the sum as, for $k \geq 2$,
	\begin{align*}
		\E \abs*{ \lambda_i \abs{\ip{a_i}{x}}^2 }^k
		&\leq \parens*{ c \frac{r}{n}}^{k-2} \E [\lambda_i^2 \abs{\ip{a_i}{x}}^{2k}] \\
		&\leq \parens*{ c \frac{r}{n}}^{k-2} (\E \lambda_i^4)^{1/2}  (\E \abs{\ip{a_i}{x}}^{4k})^{1/2} \\
		&\leq c' \parens*{ c \frac{r}{n}}^{k-2} \cdot \frac{r}{n^2} \cdot \parens*{ (c'')^{2k} \cdot (2k)! }^{1/2} \\
		&\leq c' \parens*{ c \frac{r}{n}}^{k-2} \cdot \frac{r}{n^2} \cdot k!.
	\end{align*}
	The third inequality uses a standard Gaussian moment bound (see, e.g., the proof of \Cref{lem:AZ_conc}) along with our estimate of $\E \lambda_1^4$.
	The last inequality absorbs the $(c'')^k$ term into the others and also uses the fact (e.g., by Stirling's approximation) that $\sqrt{(2k)!} \leq c' c^k k!$,
	again consolidating the constants.
	
	Then, following similar steps as in the proof of \Cref{lem:AZ_conc},
	we obtain, with probability at least $1 - 2 n^{-2}$,
	\begin{align*}
		\opnorm{Y - \E Y} &\leq c\parens*{\sqrt{\frac{r(d+\log n}{n}} + \frac{ r(d + \log n) }{n}} \\
		&\leq c \sqrt{\frac{r(d + \log n)}{n}} \eqqcolon \delta.
	\end{align*}
	Note that we can then take
	\[
		\epsilon \coloneqq \sqrt{2r} \delta \lesssim \sqrt{\frac{r^2 (d + \log n)}{n}}.
	\]
	We have only proved that, on this event, $\PTp(Y) \succeq (1 - \delta) \Ppmat$.
	However, choosing $n \gtrsim rd$ with large enough constant ensures, say, $\delta \leq 1/2$,
	so rescaling $Y$ by $(1 - \delta)^{-1} \leq 2$ gives $\PTp(Y) \succeq \PTpmat$,
	only changing the other bounds by a constant.
	This completes the proof.
\end{proof}

To prove \Cref{lem:gauss_iso},
we need the following technical lemma:
\begin{lemma}
	\label{lem:gauss_lb_fronuc}
	There exist constants $c_1, c_2 > 0$ such that, if $n \geq d$, with probability at least $1 - n^{-2}$,
	the Gaussian ensemble satisfies, for all $H \in \herms_d$,
	\[
		\frac{1}{n} \norm{\scrA(H)}_1
		\geq c_1 \normF{H} - c_2 \sqrt{\frac{d + \log n}{n}} \nucnorm{H}.
	\]
\end{lemma}
\begin{proof}
	As before, $c, c'$, etc.\ will denote universal positive constants that may change from one usage to another.
	An argument based on the Paley-Zygmund and Hanson-Wright inequalities (see, e.g., \cite[Sec.~2.8.5]{Tropp2015a} or \cite[Sec.~6]{Krahmer2020}) implies that, for all $H \in \herms_d$,
	\[
		\P( \ip{A_1}{H}^2 \geq c \normF{H}^2 ) \geq c'.
	\]
	Therefore,
	\[
		\E \frac{1}{n} \norm{\scrA(H)}_1 = \E \abs{\ip{A_1}{H}} \geq c_1 \normF{H}
	\]
	for some $c_1 > 0$.
	
	We now use a symmetrization argument to bound $\frac{1}{n} \abs{ \norm{\scrA(H)}_1 - \E \norm{\scrA(H)}_1 }$ uniformly over the unit nuclear-norm ball $B \coloneqq \{ H \in \herms_d : \nucnorm{H} \leq 1 \}$.
	A standard symmetrization inequality (e.g., \cite[Thm.~2.1]{Koltchinskii2011a}) implies that, for all $t \geq 0$,
	\begin{align*}
		&\E \exp\parens*{ t \cdot \sup_{H \in B}~\frac{1}{n}\abs*{ \norm{\scrA(H)}_1 - \E \norm{\scrA(H)}_1 } } \\
		&\qquad\leq \E \exp\parens*{ 2 t \cdot \sup_{H \in B}~\abs*{ \frac{1}{n} \sum_{i=1}^n \varepsilon_i \ip{A_i}{H} } } \\
		&\qquad= \E \exp\parens*{ 2 t \opnorm*{ \frac{1}{n} \sum_{i=1}^n \varepsilon_i A_i} },
	\end{align*}
	where $\varepsilon_1, \dots \varepsilon_n$ are i.i.d.\ Rademacher random variables independent of $\{a_i\}_{i=1}^n$.
	The last equality is by the duality of the nuclear and operator norms.
	
	By another covering argument similar to that in the proof of \Cref{lem:AZ_conc} (again, see \cite[Ch.~4]{Vershynin2018}), we have
	\begin{align*}
		&\E \exp\parens*{ 2 t \opnorm*{ \frac{1}{n} \sum_{i=1}^n \varepsilon_i A_i} } \\
		&\qquad\leq c^d \sup_{\norm{x} \leq 1}~\E \exp \parens*{ c' t \abs*{ \frac{1}{n} \sum_{i=1}^n \varepsilon_i \abs{\ip{a_i}{x}}^2  }} \\
		&\qquad\leq 2 c^d \sup_{\norm{x} \leq 1}~\E \exp \parens*{ \frac{c' t}{n} \sum_{i=1}^n \varepsilon_i \abs{\ip{a_i}{x}}^2 } \\
		&\qquad= c^d \sup_{\norm{x} = 1}~\brackets*{ \E \exp\parens*{ \frac{c' t}{n} \varepsilon_1 \abs{\ip{a_1}{x}}^2  } }^n.
	\end{align*}
	The second inequality uses the fact that $\frac{1}{n} \sum_{i} \varepsilon_i \abs{\ip{a_i}{x}}^2$ has a distribution symmetric about $0$.
	As, for all unit-norm $x \in \F^d$,
	$\varepsilon_1 \abs{\ip{a_1}{x}}^2$ is zero-mean and sub-exponential,
	we have (see, e.g., \cite[Prop.~2.7.1]{Vershynin2018}), for all $t \leq c n$,
	\[
		\E \exp\parens*{ \frac{c t}{n} \varepsilon_1 \abs{\ip{a_1}{x}}^2 } \leq \exp\parens*{ c' \frac{t^2}{n^2} }.
	\]
	Combining the last three displays gives, for $0 \leq t \leq c n$,
	\begin{align*}
		&\E \exp\parens*{ t \cdot \sup_{H \in B}~\frac{1}{n} \abs*{ \norm{\scrA(H)}_1 - \E \norm{\scrA(H)}_1 } } \\
		&\qquad \leq \exp\parens*{ c d + c' \frac{t^2}{n} },
	\end{align*}
	and, therefore, for any $c_2 > 0$, by a Chernoff bound,
	\begin{align*}
		&\P\parens*{ \sup_{H \in B}~\frac{1}{n}\abs*{ \norm{\scrA(H)}_1 - \E \norm{\scrA(H)}_1 } \geq c_2 \sqrt{\frac{d + \log n}{n}} } \\
		&\quad \leq \E \exp\parens*{ t \cdot \sup_{H \in B}~\frac{1}{n}\abs*{ \norm{\scrA(H)}_1 - \E \norm{\scrA(H)}_1 } - c_2 \sqrt{\frac{d + \log n}{n}} t } \\
		&\quad \leq \exp\parens*{ c d + c' \frac{t^2}{n} - c_2 \sqrt{\frac{d + \log n}{n}} t }.
	\end{align*}
	As $n \geq d$, we can choose $t = c \sqrt{n(d + \log n)} \leq c n$,
	and, if $c_2 > 0$ is sufficiently large, we obtain
	\begin{align*}
		&\P\parens*{ \sup_{H \in B}~\frac{1}{n}\abs*{ \norm{\scrA(H)}_1 - \E \norm{\scrA(H)}_1 } \geq c_2 \sqrt{\frac{d + \log n}{n}} } \\
		&\qquad	\leq \exp\parens*{ - 2 \log n }
		= n^{-2}.
	\end{align*}
	As the claim is invariant to rescaling of $H$, this completes the proof.
\end{proof}

\begin{proof}[Proof of \Cref{lem:gauss_iso}]
	Let $H \in \herms_d$.
	Note that
	\begin{align*}
		\nucnorm{H} &\leq \nucnorm{\PT(H)} + \nucnorm{\PTp(H)} \\
		&\leq \sqrt{2r} \normF{\PT(H)} + \nucnorm{\PTp(H)}.
	\end{align*}
	\Cref{lem:gauss_lb_fronuc} implies that, with probability at least $1 - n^{-2}$,
	for all $H \in \herms_d$,
	\begin{align*}
		\frac{1}{\sqrt{n}} \norm{\scrA(H)}
		&\geq \frac{1}{n} \norm{\scrA(H)}_1 \\
		&\geq c_1 \normF{H} - c_2 \sqrt{\frac{d + \log n}{n}} \nucnorm{H}
	\end{align*}
	for some constants $c_1, c_2 > 0$ that will remain fixed for the rest of this proof.
	On this event, we then have
	\begin{align*}
		\frac{1}{\sqrt{n}} \norm{\scrA(H)}
		&\geq c_1 \normF{\PT(H)} - c_2 \sqrt{\frac{d + \log n}{n}} \nucnorm{H} \\
		&\geq \parens*{c_1 - c_2 \sqrt{\frac{2r(d + \log n)}{n}}} \normF{\PT(H)} \\
		&\qquad - c_2 \sqrt{\frac{d + \log n}{n}} \nucnorm{\PTp(H)}.
	\end{align*}
	With $n \gtrsim rd$, we have
	\[
		\loweriso \coloneqq c_1 - c_2 \sqrt{\frac{2r(d + \log n)}{n}} \gtrsim 1,
	\]
	and we set
	\[
		\upperiso \coloneqq c_2 \sqrt{\frac{d + \log n}{n}}.
	\]
	This completes the proof.
\end{proof}

\section*{Acknowledgements}
\ifMS \else \fundingack{} \fi
The author thanks Jonathan Dong, Richard Y.\ Zhang, Nicolas Boumal, Christopher Criscitiello, Andreea Muşat, and Quentin Rebjock for helpful inspiration, discussions, and suggestions.
The author also thanks the two anonymous reviewers whose constructive feedback contributed to many improvements in the paper.

\bibliographystyle{IEEEtran}
\bibliography{refs_pr_ncvx}

\ifMS
\begin{IEEEbiographynophoto}{Andrew D. McRae}
	received B.S., M.S., and Ph.D.\ degrees from the Georgia Institute of Technology in 2015, 2016, and 2022, respectively.
	He is currently an assistant professor with the CERMICS group at the École nationale des ponts et chaussées (ENPC),
	part of the Institut Polytechnique de Paris.
	Prior to this, in 2022-2025, he was a postdoctoral researcher
	with the Institute of Mathematics at the École Polytechnique Fédérale de Lausanne (EPFL).
	His research interests are in high-dimensional statistics and signal processing with a particular focus on the role of	optimization.
\end{IEEEbiographynophoto}
\fi

\end{document}

%% file: definitions.tex
\usepackage{amsthm}
\usepackage{xparse}
\usepackage{import}
\usepackage[normalem]{ulem}

\makeatletter
\@ifpackageloaded{unicode-math}{%
	\newcommand{\mymathbold}{\symbf}%
}{%
	\usepackage{bm}%
	\newcommand{\mymathbold}{\bm}%
}
\makeatother

\subimport{./}{symbols.tex}

\DeclareMathOperator{\E}{\mathbf{E}}
\renewcommand{\P}{\operatorname{\mathbf{P}}}

\newcommand{\rank}{\operatorname{rank}}

\newcommand{\tr}{\operatorname{tr}}

\newcommand{\ddiag}{\operatorname{ddiag}}

\DeclarePairedDelimiter{\norm}{\lVert}{\rVert}

\DeclarePairedDelimiter{\abs}{\lvert}{\rvert}

\DeclarePairedDelimiter{\braces}{\{}{\}}
\DeclarePairedDelimiter{\parens}{(}{)}
\DeclarePairedDelimiter{\brackets}{[}{]}
\DeclarePairedDelimiterX{\ip}[2]{\langle}{\rangle}{#1,#2}

\DeclarePairedDelimiterXPP{\normsub}[2]{}{\lVert}{\rVert}{_{#2}}{#1}
\DeclarePairedDelimiterXPP{\ipsub}[3]{}{\langle}{\rangle}{_{#3}}{#1,#2}

%Create custom norm/inner product commands with something like
%\usepackage{suffix}
%...
%\newcommand{\iptwo}[2]{\ipsub{#1}{#2}{2}}
%\WithSuffix\newcommand{\iptwo*}[2]{\ipsub*{#1}{#2}{2}}

\DeclarePairedDelimiterXPP{\ipHS}[2]{}{\langle}{\rangle}{_{\mathrm{HS}}}{#1, #2}
\DeclarePairedDelimiterXPP{\normHS}[1]{}{\lVert}{\rVert}{_{\mathrm{HS}}}{#1}

\DeclarePairedDelimiterXPP{\ipF}[2]{}{\langle}{\rangle}{_{\mathrm{F}}}{#1, #2}
\DeclarePairedDelimiterXPP{\normF}[1]{}{\lVert}{\rVert}{_{\mathrm{F}}}{#1}

\DeclarePairedDelimiterXPP{\dkl}[2]{\operatorname{D_{KL}}}{(}{)}{}{#1 \: \delimsize\Vert \: #2}

\DeclarePairedDelimiterXPP{\restr}[2]{}{{}}{\vert}{_{#2}}{#1}

\newcommand{\R}{\mathbf{R}}
\newcommand{\C}{\mathbf{C}}

\newcommand{\F}{\mathbf{F}}

\newcommand{\indicator}[1]{\mathbf{1}_{\{ #1 \}}}
\newcommand{\range}{\operatorname{range}}

 % Avoid error with Beamer

%% file: symbols.tex
% Define all the modified symbols one could think of.
% Adapted from https://tex.stackexchange.com/questions/477180/loop-over-greek-letters-with-pgffor

% \blA = bold uppercase A. \opA = operator, \scrA = script, \Ahat = hat, \Abr = bar, \Atl = tilde, \Adot = dot
% Lowercase: \bla, etc. No operator or script.
% Greek: same as Roman with \blGamma, \alphahat, \Xitl, etc.

%\newcommand{\op}[1]{\operatorname{\symcal{#1}}}
%\newcommand{\opbar}[1]{\operatorname{\overline{\symcal{#1}}}}
\newcommand{\scrbar}[1]{\overline{\mathcal{#1}}}
\newcommand{\scrhat}[1]{\widehat{\mathcal{#1}}}
\newcommand{\scrtl}[1]{\widetilde{\mathcal{#1}}}

\ExplSyntaxOn

\cs_new_protected:Nn \bamboo_define:nnnnnN
% #1 = name, #2 = argument, #3 = prefix, #4 = suffix, #5 = decoration, #6 = how to interpret decoration
{
	\cs_new_protected:cpx { #3 #1 #4 } { \exp_not:N #5{#6{#2}} }
}

\int_step_inline:nnn { `A } { `Z }
{
	% Bold
	\bamboo_define:nnnnnN
	{ \char_generate:nn { #1 } { 11 } }
	{ \char_generate:nn { #1 } { 11 } }
	{ bl }
	{ }
	{ \mymathbold }
	\use:n
	
%	% Operator
%	\bamboo_define:nnnnnN
%	{ \char_generate:nn { #1 } { 11 } }
%	{ \char_generate:nn { #1 } { 11 } }
%	{ op }
%	{ }
%	{ \op }
%	\use:n
	
	% Script
	\bamboo_define:nnnnnN
	{ \char_generate:nn { #1 } { 11 } }
	{ \char_generate:nn { #1 } { 11 } }
	{ scr }
	{ }
	{ \mathcal }
	\use:n
	
	% Hat
	\bamboo_define:nnnnnN
	{ \char_generate:nn { #1 } { 11 } }
	{ \char_generate:nn { #1 } { 11 } }
	{ }
	{ hat }
	{ \widehat }
	\use:n
	
	% Script and hat
	\bamboo_define:nnnnnN
	{ \char_generate:nn { #1 } { 11 } }
	{ \char_generate:nn { #1 } { 11 } }
	{ scr }
	{ hat }
	{ \scrhat }
	\use:n
	
	% Bar
	\bamboo_define:nnnnnN
	{ \char_generate:nn { #1 } { 11 } }
	{ \char_generate:nn { #1 } { 11 } }
	{ }
	{ br }
	{ \overline }
	\use:n
	
	% Tilde
	\bamboo_define:nnnnnN
	{ \char_generate:nn { #1 } { 11 } }
	{ \char_generate:nn { #1 } { 11 } }
	{ }
	{ tl }
	{ \widetilde }
	\use:n
	
	% Script with bar
	\bamboo_define:nnnnnN
	{ \char_generate:nn { #1 } { 11 } }
	{ \char_generate:nn { #1 } { 11 } }
	{ scr }
	{ br }
	{ \scrbar }
	\use:n
	
	% Script with tilde
	\bamboo_define:nnnnnN
	{ \char_generate:nn { #1 } { 11 } }
	{ \char_generate:nn { #1 } { 11 } }
	{ scr }
	{ tl }
	{ \scrtl }
	\use:n
	
	% Dot
	\bamboo_define:nnnnnN
	{ \char_generate:nn { #1 } { 11 } }
	{ \char_generate:nn { #1 } { 11 } }
	{ }
	{ dt }
	{ \dot}
	\use:n
}
\int_step_inline:nnn { `a } { `z }
{
	% Bold
	\bamboo_define:nnnnnN
	{ \char_generate:nn { #1 } { 11 } }
	{ \char_generate:nn { #1 } { 11 } }
	{ bl }
	{ }
	{ \mymathbold }
	\use:n
	
	% Hat
	\bamboo_define:nnnnnN
	{ \char_generate:nn { #1 } { 11 } }
	{ \char_generate:nn { #1 } { 11 } }
	{ }
	{ hat }
	{ \hat }
	\use:n
	
	% Bar
	\bamboo_define:nnnnnN
	{ \char_generate:nn { #1 } { 11 } }
	{ \char_generate:nn { #1 } { 11 } }
	{ }
	{ br }
	{ \bar }
	\use:n
	
	% Tilde
	\bamboo_define:nnnnnN
	{ \char_generate:nn { #1 } { 11 } }
	{ \char_generate:nn { #1 } { 11 } }
	{ }
	{ tl }
	{ \tilde }
	\use:n                            % just the argument
	
	% Dot
	\bamboo_define:nnnnnN
	{ \char_generate:nn { #1 } { 11 } }
	{ \char_generate:nn { #1 } { 11 } }
	{ }
	{ dt }
	{ \dot}
	\use:n
}
\clist_map_inline:nn
{
	Gamma,Delta,Theta,Lambda,Xi,Pi,Sigma,Phi,Psi,Omega
}
{
	% Bold
	\bamboo_define:nnnnnN
	{ #1 }
	{ #1 }
	{ bl }
	{ }
	{ \mymathbold }
	\use:c
	
	% Operator
	\bamboo_define:nnnnnN
	{ #1 }
	{ #1 }
	{ op }
	{ }
	{ \op }
	\use:c
	
	% Script
	\bamboo_define:nnnnnN
	{ #1 }
	{ #1 }
	{ scr }
	{ }
	{ \mathcal }
	\use:c
	
	% Hat
	\bamboo_define:nnnnnN
	{ #1 }
	{ #1 }
	{ }
	{ hat }
	{ \widehat }
	\use:c
	
	% Bar
	\bamboo_define:nnnnnN
	{ #1 }
	{ #1 }
	{ }
	{ br }
	{ \overline }
	\use:c
	
	% Tilde
	\bamboo_define:nnnnnN
	{ #1 }
	{ #1 }
	{ }
	{ tl }
	{ \widetilde }
	\use:c
	
	% Dot
	\bamboo_define:nnnnnN
	{ #1 }
	{ #1 }
	{ }
	{ dt }
	{ \dot}
	\use:c
	
}
\clist_map_inline:nn
{
	alpha,beta,gamma,delta,epsilon,zeta,eta,theta,iota,kappa,
	lambda,mu,nu,xi,pi,rho,sigma,tau,phi,chi,psi,omega
}
{
	% Bold
	\bamboo_define:nnnnnN
	{ #1 }
	{ #1 }
	{ bl }
	{ }
	{ \mymathbold }
	\use:c
	
	% Hat
	\bamboo_define:nnnnnN
	{ #1 }
	{ #1 }
	{ }
	{ hat }
	{ \hat }
	\use:c
	
	% Bar
	\bamboo_define:nnnnnN
	{ #1 }
	{ #1 }
	{ }
	{ br }
	{ \bar }
	\use:c
	
	% Tilde
	\bamboo_define:nnnnnN
	{ #1 }
	{ #1 }
	{ }
	{ tl }
	{ \tilde }
	\use:c
	
	% Dot
	\bamboo_define:nnnnnN
	{ #1 }
	{ #1 }
	{ }
	{ dt }
	{ \dot}
	\use:c
}

\ExplSyntaxOff